\documentclass[a4paper,11pt]{article}
\usepackage[utf8]{inputenc}
\usepackage[T1]{fontenc}
\usepackage{amsmath, amsthm, amssymb}
\usepackage{physics} 
\usepackage{braket} 
\usepackage{enumitem} 
\usepackage{esint} 
\usepackage{graphicx} 
\usepackage[toc,page]{appendix} 
\usepackage{hyperref} 

\newcommand{\mres}{\mathbin{\vrule height 1.6ex depth 0pt width 0.13ex\vrule
height 0.13ex depth 0pt width 0.8ex}}

\theoremstyle{plain}
\newtheorem{thm}{Theorem}[section]
\newtheorem{prop}[thm]{Proposition}
\newtheorem{cor}[thm]{Corollary}
\newtheorem{lem}[thm]{Lemma}
\newtheorem*{thm*}{Theorem}
\newtheorem*{prop*}{Proposition}
\newtheorem*{cor*}{Corollary}
\newtheorem*{lem*}{Lemma}

\theoremstyle{definition}
\newtheorem{defi}[thm]{Definition}

\newtheorem*{defi*}{Definition}
\newtheorem*{exa*}{Example}

\theoremstyle{remark}
\newtheorem{rmk}[thm]{Remark}
\newtheorem*{rmk*}{Remark}

\newcommand{\R}{\mathbf{R}}

\newcommand{\HH}{\mathcal{H}}
\newcommand{\LL}{\mathcal{L}}
\newcommand{\II}{\mathcal{I}}
\newcommand{\FF}{\mathcal{F}}
\newcommand{\spt}{\mathrm{spt}}
\newcommand{\BV}{\mathrm{BV}}
\newcommand{\SBV}{\mathrm{SBV}}
\newcommand{\loc}{\mathrm{loc}}
\newcommand{\dm}{\,\mathrm{d}}

\title{Higher integrability of the gradient for the Thermal Insulation problem}

\author{C. Labourie, E. Milakis}

\date{}

\begin{document}

\maketitle

\begin{abstract}
    We prove the higher integrability of the gradient for local minimizers of
    the thermal insulation problem, an analogue of De Giorgi's conjecture for
    the Mumford-Shah functional. We deduce that the singular part of the free
    boundary has Hausdorff dimension strictly less than $n-1$.
\end{abstract}



\tableofcontents

\section{Introduction}

We fix a bounded connected set $\Omega \subset \R^n$.
The thermal insulation problem consists in minimizing the functional
\begin{equation}
    \II(A,u) := \int_{A} \! \abs{\nabla u}^2 \dm\LL^n + \int_{\partial A} \! \abs{u^*}^2 \dm\HH^{n-1} + \LL^n(A)
\end{equation}
among all pairs $(A,u)$ where $A \subset \R^n$ is an admissible domain and $u \in W^{1,2}(A)$ is a function such that $u = 1$ for $\LL^n$-a.e.\ on $\Omega$.
Here, $u^*$ is the trace of $u$ on $\partial A$.

The problem has been studied by Caffarelli--Kriventsov in~\cite{CK},~\cite{Kriventsov} and Bucur--Giacomini--Luckhaus in~\cite{BL},~\cite{BG}.
The authors transpose the problem to a slightly different setting in order to apply the direct method of the calculus of variation.
The authors represent a pair $(A,u)$ by the function $u \mathbf{1}_A$ and relax the functional on $\SBV$.
The new problem consists in minimizing the functional
\begin{equation}\label{eq_problem}
    \FF(u) := \int_{\R^n} \! \abs{\nabla u}^2 \dm\LL^n + \int_{J_u} \! (\overline{u}^2 + \underline{u}^2) \dm\HH^{n-1} + \LL^n(\set{u > 0})
\end{equation}
among all functions $u \in \SBV(\R^n)$ such that $u=1$ $\LL^n$-a.e.\ on $\Omega$.
The definition of $J_u$ and $\overline{u}$, $\underline{u}$ are given in Appendix~\ref{appendix_bv}.
This new setting is more suited to a direct minimization since it enjoys the compactness and closure properties of $\SBV$.
In parenthesis, there always exist functions $u \in \SBV(\R^n)$ such that $u=1$ $\LL^n$-a.e.\ on $\Omega$ and $\FF(u) < \infty$.
For example, $u := \mathbf{1}_B$ where $B$ is an open ball containing $\Omega$.
In~\cite[Theorem 4.2]{CK}, Caffarelli and Kriventsov prove that the $\SBV$ problem has a solution $u$.
A key point property of solutions is that there exists $0 < \delta < 1$ (depending on $n$, $\Omega$) such that $\spt(u) \subset B(0,\delta^{-1})$ and
\begin{equation}
    u \in \set{0} \cup [\delta,1] \quad \text{$\LL^n$-a.e.\ on $\R^n$}.
\end{equation}
This property has also been proved in~\cite{BL}. On another note, some minimality criteria have been proved by calibrations in~\cite{LM}.

The main goal of the present article is to prove that there exists $p > 1$ such that $\abs{\nabla u}^2 \in L^p_{\loc}(\R^n \setminus \overline{\Omega})$ (Theorem~\ref{thm_integrability}).
A parallel property was conjectured by De Giorgi for minimizers of the Mumford-Shah functional and solved by De Lellis and Focardi in the planar case (\cite{DLF}) and then De Philippis and Figalli (\cite{DPF}) in the general case.
Our proof is inspired by the technique of~\cite{DPF} and it relies on three key properties: the Ahlfors-regularity of the free boundary, the uniform rectifiability of the free boundary and the $\varepsilon$-regularity theorem.
In particular, this implies a porosity property which means that the singular part $\Sigma$ of the free boundary has many holes in a quantified way.
In contrast to the Mumford-Shah situation, the $\varepsilon$-regularity theorem describes a regular part of the boundary as a pair of graphs rather than just one graph.
The minimizers satisfies an elliptic equation with a Robin boundary condition at the boundary rather than a Neumann boundary condition.
We present the technique of~\cite{DPF} in a different way by singling out a higher integrability lemma and a covering lemma and by removing the need of~\cite[Lemma 3.2]{DPF} (the existence of good radii).
Once we establish the higher integrability of the gradient, we are also able to conclude that the dimension of $\Sigma$ is strictly less than $n-1$ (Theorem~\ref{thm_dimension}).
The link between the higher integrability of the gradient and the dimension of the singular part has been first observed for the Mumford-Shah functional by Ambrosio, Fusco, Hutchinson in~\cite{AFH}.
An open question of Caffarelli--Kriventsov hints that for all minimizers in the planar case, $\Sigma$ is empty and the optimal exponent is $p=\infty$ (see also Remark~\ref{rmk_dimension}).

\textbf{Acknowledgement}. We would like to thank Guido De Philippis for his helpful correspondence concerning estimates for elliptic equations.

\section{Generalities about minimizers}

\subsection{Definition}

\textbf{Notations}. 
Our ambient space is an open set $X$ of $\R^n$.
One can think of $X$ as $\R^n \setminus \overline{\Omega}$.
For $x \in \R^n$ and $r > 0$, $B(x,r)$ is the \emph{open ball} centered in $x$ and of radius $r$. If there is no ambiguity, it is simply denoted by $B_r$.
Given an open ball $B := B(x,r)$ and a scalar $t > 0$, the notation $t B$ means $B(x,t r)$.
Given a set $A \subset \R^n$, the \emph{indicator function} of $A$ is denoted by $\mathbf{1}_A$.
Given two sets $A, B \subset \R^n$, the notation $A \subset \subset B$ means that there exists a compact set $K \subset \R^n$ such that $A \subset K \subset B$.

Given $u \in \SBV_{\loc}(X)$, we denote by $K$ the support of the singular part of $Du$:
\begin{subequations}
    \begin{align}
        K   &:= \spt(\abs{\overline{u}-\underline{u}} \HH^{n-1} \mres J_u)\\
            &:= \spt(\HH^{n-1} \mres J_u).
    \end{align}
\end{subequations}
For $x \in K$ and $r > 0$ such that $\overline{B}(x,r) \subset X$, we define
\begin{subequations}
    \begin{align}
        \omega_2(x,r)       &:= r^{-(n-1)} \int_{B(x,r)} \! \abs{\nabla u}^2 \dm\LL^n,\\
        \beta_2(x,r)        &:= \left(r^{-(n+1)} \inf_V \int_{K \cap \overline{B}(x,r)} \! \dm(y,V)^2 \dm\HH^{n-1}(y)\right)^{\frac{1}{2}},
    \end{align}
\end{subequations}
where $V$ runs among $(n-1)$ planes $V \subset \R^n$ passing through $x$.
When there is an ambiguity, we will write $\beta_{K,2}$ instead of $\beta_2$.
We have gathered some definitions and results from the theory of $\BV$ functions in the introduction of Appendix~\ref{appendix_bv}.

For any open ball $B$ such that $\overline{B} \subset X$, we define a \emph{competitor of $u$ in $B$} as a function $v \in \SBV_{\loc}(X)$ such that $v = u$ $\LL^n$-a.e.\ on $X \setminus \overline{B}$.
We fix a constant $\delta \in \left]0,1\right[$ for all the paper.

\begin{defi}\label{defi_mini}
    We say that $u \in \SBV_{\loc}(X)$ is a \emph{local minimizer} if
    \begin{enumerate}
        \item
            for $\LL^n$-a.e.\ $x \in X$, we have $u \in \set{0} \cup [\delta,\delta^{-1}]$;

        \item
            for all open balls $B$ such that $\overline{B} \subset X$, for all competitors $v$ of $u$ in $B$,
            \begin{multline}
                \int_B \! \abs{\nabla u}^2 \dm\LL^n + \int_{J_u \cap \overline{B}} \! (\overline{u}^2 + \underline{u}^2) \dm\HH^{n-1} + \LL^n(\set{u > 0} \cap B)\\
                \leq \int_B \! \abs{\nabla v}^2 \dm\LL^n + \int_{J_v \cap \overline{B}} \! (\overline{v}^2 + \underline{v}^2) \dm\HH^{n-1} + \LL^n(\set{v > 0} \cap B).
            \end{multline}
    \end{enumerate}
\end{defi}

As a first consequence, we have that $\overline{u}, \underline{u} \in \set{0} \cup [\delta,\delta^{-1}]$ \emph{everywhere} in $X$.
In particular, $\overline{u} \geq \delta$ everywhere on $S_u$.
For all open balls $B$ such that $\overline{B} \subset X$, we have
\begin{equation}
    \int_B \! \abs{\nabla u}^2 \dm\LL^n + \int_{J_u \cap \overline{B}} \! (\overline{u}^2 + \underline{u}^2) \dm\HH^{n-1} < \infty.
\end{equation}
This shows that $\abs{\nabla u}^2 \in L^1_{\loc}(X)$ and that $S_u$ is $\HH^{n-1}$-locally finite in $X$.
In $X \setminus \overline{S_u}$, the function $u$ is $W_{\loc}^{1,2}$ and locally minimizes its Dirichlet energy.
Therefore, $u$ is harmonic (and thus continuous) in $X \setminus \overline{S_u}$.
We conclude that in each connected component of $X \setminus \overline{S_u}$, we have either $u > \delta$ everywhere or $u = 0$ everywhere.

\subsection{Properties}

The next results (Ahlfors-regularity, uniform rectifiability and $\varepsilon$-regularity theorem) also hold true for the almost-minimizers of \cite[Definition 2.1]{Kriventsov}.
We are going to cite~\cite[Corollary 3.3 and Theorem 5.1]{CK}.

\begin{prop}[Ahlfors-regularity]\label{prop_af}
    Let $u \in \SBV_{\loc}(X)$ be a local minimizer.
    There exists $0 < r_0 \leq 1$ and $C \geq 1$ (both depending on $n$, $\delta$) such that the following holds true.
    \begin{enumerate}
        \item
            For all $x \in X$, for all $0 \leq r \leq r_0$ such that $B(x,r) \subset X$,
            \begin{equation}\label{eq_af1}
                \int_{B(x,r)} \! \abs{\nabla u}^2 \dm \LL^n + \HH^{n-1}(K \cap B(x,r)) \leq C r^{n-1}.
            \end{equation}

        \item
            For all $x \in \overline{S_u}$, for all $0 \leq r \leq r_0$ such that $B(x,r) \subset X$,
            \begin{equation}\label{eq_af2}
                \HH^{n-1}(K \cap B(x,r)) \geq C^{-1} r^{n-1}.
            \end{equation}
    \end{enumerate}
\end{prop}

\begin{cor}
    Let $u \in \SBV_{\loc}(X)$ be a local minimizer.
    \begin{enumerate}[label = (\roman*)]
        \item
            We have $K = \overline{S_u} = \overline{J_u}$ and $\HH^{n-1}(K \setminus J_u) = 0$.

        \item
            The set $A_u:= \set{\overline{u} > 0} \setminus K$ is open and $\partial A_u = K$.
    \end{enumerate}
\end{cor}

\begin{proof}
    It is straighforward by definition that $K \subset \overline{J_u} \subset \overline{S_u}$.
    On the other hand, property (\ref{eq_af2}) shows that $\overline{S_u} \subset K$.
    We justify that $\HH^{n-1}(K \setminus J_u) = 0$.
    The jump set $J_u$ is Borel and $\HH^{n-1}$ locally finite in $X$, so for $\HH^{n-1}$-a.e.\ $x \in X \setminus J_u$,
    \begin{equation}
        \lim_{r \to 0} \frac{\HH^{n-1}(J_u \cap B(x,r))}{r^{n-1}} = 0
    \end{equation}
    (see~\cite[Theorem $6.2$]{Mattila}).
    We draw our claim from the observation that this limit contradicts (\ref{eq_af2}).

    We study the set $A_u$.
    We recall that the function $\overline{u}$ is continuous in $X \setminus K$ (since it coincides with $u$ outside $S_u$) and $\overline{u} \in \set{0} \cup [\delta,1]$ \emph{everywhere} in $X \setminus K$.
    As a consequence, the sets
    \begin{align}
        A_u &:= \set{\overline{u} > 0} \setminus K,\\
        B_u &:= \set{\overline{u} = 0} \setminus K
    \end{align}
    are open subsets of $X \setminus K$ and thus of $X$.
    The space $X$ is the disjoint union
    \begin{equation}
        X = K \cup A_u \cup B_u,
    \end{equation}
    where $A_u$ and $B_u$ are open and $K$ is relatively closed, so $\overline{A_u} \subset A_u \cup K$.

    We show that $S_u \subset \overline{A_u}$.
    Let us suppose that there exists $x \in S_u$ and $r > 0$ such that $A_u \cap B(x,r) = \emptyset$.
    Then $B(x,r) \setminus K \subset \set{\overline{u} = 0}$ so we have $u=0$ $\LL^n$-a.e.\ on $B(x,r)$ and thus $x$ is a Lebesgue point of $u$ (a contradiction).
    We conclude that $S_u \subset \overline{A_u}$ and in turn $K \subset \overline{A_u}$ so $\overline{A_u} = A_u \cup K$.
\end{proof}

We are going to apply~\cite{David} to justify that $K$ is locally contained in a uniformly rectifiable set.
We underline that our local minimizers are \emph{not} quasiminimizers as in \cite[Definition 7.21]{David}.
In Appendix~\ref{appendix_david}, we have summarised the relevant results of~\cite{David} and how their proofs adapt to our case (Remark~\ref{rmk_david2}).

\begin{prop}[Uniform Rectifiability]\label{prop_ur}
    Let $u \in \SBV_{\loc}(X)$ be a local minimizer.
    There exists $0 < r_0 \leq 1$ (depending on $n$, $\delta$) such that the following holds true.
    For all $x \in K$ and $0 < r \leq r_0$ such that $B(x,r) \subset X$, there is a closed, Ahlfors-regular, uniformly rectifiable set $E$ of dimension $(n-1)$ such that $K \cap \tfrac{1}{2}B(x,r) \subset E$.
    The constants for the Ahfors-regularity and uniform rectifiability depends on $n$, $\delta$.
\end{prop}

\begin{proof}
    We want to show that $(u,K)$ satisfy Definition~\ref{defi_david_quasi}, or rather the alternative Definition given in Remark~\ref{rmk_david2}.
    Then the Proposition will follow from Theorem~\ref{thm_david}.
    First, it is clear that $(u,K)$ is an admissible pair.
    Let $B$ be an open ball of radius $r > 0$ such that $\overline{B} \subset X$.
    Let an admissible pair $(v,L)$ be a competitor of $(u,K)$ in $B$.
    As explained in Remark~\ref{rmk_david}, we can assume without loss of generality that $L$ is $\HH^{n-1}$ locally
    finite.
    Therefore, $v \in \SBV_{\loc}(X)$ and $\HH^{n-1}(J_v \setminus L) = 0$.
    We have included more details about the construction of $\SBV$ functions in Appendix~\ref{appendix_bv}.
    We can now apply the minimality inequality.
    We have
    \begin{multline}
        \int_B \! \abs{\nabla u}^2 \dm\LL^n + \int_{J_u \cap \overline{B}} \! (\overline{u}^2 + \underline{u}^2) \dm\HH^{n-1} + \LL^n(\set{u > 0} \cap B)\\
        \leq \int_B \! \abs{\nabla v}^2 \dm\LL^n + \int_{J_v \cap \overline{B}} \! (\overline{v}^2 + \underline{v}^2) \dm\HH^{n-1} + \LL^n(\set{v > 0} \cap B)
    \end{multline}
    so
    \begin{multline}
        \int_B \! \abs{\nabla u}^2 \dm\LL^n + \delta^2 \HH^{n-1}(J_u \cap \overline{B}) + \LL^n(\set{u > 0} \cap B) \\
        \leq \int_B \! \abs{\nabla v}^2 \dm\LL^n + \delta^{-2} \HH^{n-1}(J_v \cap \overline{B}) + \LL^n(\set{v > 0} \cap B).
    \end{multline}
    We omit the term $\LL^n(\set{u > 0} \cap B)$ at the left and we bound the term $\LL^n(\set{v > 0} \cap B)$ at the right by $\omega_n r^n$ where $\omega_n$ is the Lebesgue volume of the unit ball.
    We can replace $J_u$ by $K$ at the left since $\HH^{n-1}(K \setminus J_u) = 0$.
    We can replace $J_v$ by $L$ at the right since $\HH^{n-1}(J_v \setminus L) = 0$.
    It follows that
    \begin{equation}
        \HH^{n-1}(K \cap \overline{B}) \leq \delta^{-4} \HH^{n-1}(L \cap \overline{B}) + \delta^{-2} \Delta E + \delta^{-2} \omega_n r^n
    \end{equation}
    where
    \begin{equation}
        \Delta E := \int_B \! \abs{\nabla v}^2 - \int_B \! \abs{\nabla u}^2 \dm\LL^n.
    \end{equation}
\end{proof}

We are going to cite the $\varepsilon$-regularity theorem for our problem~\cite[Theorem 14.1]{Kriventsov}.
Contrary to the $\varepsilon$-regularity theorem for the Mumford-Shah problem, it does not require $\omega_2(x,r)$ to be small.
It says that when $K$ is very close to a plane, $K$ is given by a pair of smooth graphs.
We describe this situation in the next definition.

Given a point $x \in \R^n$, a vector $e_n \in \mathbf{S}^{n-1}$, we can decompose each point $y \in \R^n$ under the form $y = x + (y' + y_n e_n)$ where $y' \in e_n^\perp$ and $y_n \in \R$.
Then for all function $f \colon e_n^\perp \to \R$, we define the graph of $f$ in the coordinate system $(x,e_n)$ as
\begin{equation}\label{eq_gamma}
    \Gamma_{(x,e_n)}(f) := \set{y \in \R^n | y_n = f(y')}.
\end{equation}

\begin{defi}\label{defi_regular}
    Let $u \in \SBV_{\loc}(X)$ be a local minimizer.
    Let $x \in K$ and $R > 0$ be such that $B(x,R) \subset X$.
    Let $0 < \alpha \leq 1$.
    We say that $K$ is \emph{$C^{1,\alpha}$-regular} in $B := B(x,R) \subset X$ if it satisfies the three following conditions.
    \begin{enumerate}[label = (\roman*)]
        \item
            There exists a vector $e_n \in \mathbf{S}^{n-1}$ and two functions $f_i \colon e_n^\perp \to \R$ $(i=1,2)$ such that $f_1 \leq f_2$ and
            \begin{equation}
                K \cap B = \left(\bigcup_{i=1,2} \Gamma_{(x,e_n)}(f_i)\right) \cap B.
            \end{equation}
            The functions $f_1, f_2$ are $C^{1,\alpha}$ and
            \begin{equation}
                R^{-1} \abs{f_i}_\infty + \abs{\nabla f_i}_\infty + R^\alpha \left[\nabla f_i\right]_{\alpha} \leq \frac{1}{4}.
            \end{equation}

        \item
            There are two possible cases.
            The first case is
            \begin{equation}\label{eq_first_case}
                \begin{cases}
                    u > 0 & \text{in $\set{y \in B | y_n < f_1(y') \ \text{or} \ y_n > f_2(y')}$}\\
                    u = 0 & \text{in $\set{y \in B | f_1(y') < y_n < f_2(y')}$}
                \end{cases}
            \end{equation}
            The second case is $f_1 = f_2$ and
            \begin{equation}\label{eq_second_case}
                \begin{cases}
                    u > 0 & \text{in $\set{y \in B | y_n > f_1(y')}$}\\
                    u = 0 & \text{in $\set{y \in B | y_n < f_1(y')}$}
                \end{cases}
            \end{equation}
            or inversely.
    \end{enumerate}
\end{defi}

\begin{thm}[$\varepsilon$-regularity theorem]\label{thm_regularity}
    Let $u \in \SBV_{\loc}(X)$ be a local minimizer and let $x \in K$.
    \begin{enumerate}[label = (\roman*)]
        \item
            For all $\varepsilon > 0$, there exists $\varepsilon_1 > 0$ (depending on $n$, $\delta$, $\beta$) such that the following holds true.
            For $r > 0$ such that $\overline{B}(x,r) \subset X$ and $\beta_2(x,r) + r \leq \varepsilon_1$, we have $\omega_2(x,\tfrac{r}{2}) \leq \varepsilon$.

        \item
            There exists $\varepsilon > 0$, $C \geq 1$ and $0 < \alpha < 1$ (both depending on $n$, $\delta$) such that the following holds true.
            For $r > 0$ such that $\overline{B}(x,r) \subset X$ and $\beta_2(x,r) + r \leq \varepsilon$, the set $K$ is $C^{1,\alpha}$-regular in $B(x,C^{-1}R)$.
    \end{enumerate}
\end{thm}

This last result is specific to local minimizers and does not hold true for the general almost-minimizers of \cite[Definition 2.1]{Kriventsov}.

\begin{prop}
    Let $u \in \SBV_{\loc}(X)$ be a local minimizer.
    Let $x \in K$ and $R > 0$ be such that $B(x,R) \subset X$.
    We assume that $K$ is regular in $B := B(x,R)$ (Definition~\ref{defi_regular}) and we denote by $\Gamma_i$ the graph of $f_i$ in $B$ ($i=1,2$).
    Then for each $i=1,2$, $u_{|A_i}$ solves the Robin problem
    \begin{equation}\label{eq_robin0}
        \left\{
            \begin{array}{lcl}
                \Delta u                    & = & 0 \quad \text{in $A_i$}\\
                \partial_{\nu_i}u - u_i     & = & 0 \quad \text{in $\Gamma_i$},
            \end{array}
        \right.
    \end{equation}
    where
    \begin{align}
        A_1 &:= \set{y \in B | y_n < f_1(y')}\\
        A_2 &:= \set{y \in B | y_n > f_2(y')}
    \end{align}
    and $\nu_i$ is the inner normal vector to $A_i$.
\end{prop}

\begin{proof}
    We only detail the case (\ref{eq_first_case}) of Definition~\ref{defi_regular} and we prove the Proposition for $i=2$.
    We partition $B$ in three sets (modulo $\LL^n$)
    \begin{align}
        A_1 &:= \set{z \in B | y_n < f_1(y')}\\
        A_2 &:= \set{z \in B | y_n > f_2(y')}\\
        A_3 &:= \set{z \in B | f_1(y') < x_n < f_2(y')}.
    \end{align}
    The first paragraph is devoted to detail a few generalities about traces and upper/lower limits.
    We consider a general $v \in L^\infty(B) \cap W_{\loc}^{1,2}(B \setminus K)$ such that $v = 0$ in $A_3$.
    For each $i = 1,2$, there exists $v_i^* \in L^1(\Gamma_i)$ such that for $\HH^{n-1}$-a.e.\ $x \in \Gamma_i$,
    \begin{equation}\label{eq_trace_def0}
        \lim_{r \to 0} r^{-n} \int_{A_i \cap B(x,r)} \! \abs{v(y) - v_i^*(x)} \dm\LL^n(y) = 0.
    \end{equation}
    The boundary $\Gamma_i$ is $C^1$ so for all $x \in \Gamma_i$, there a vector $\nu_i(x) \in \mathbf{S}^{n-1}$ such that
    \begin{equation}
        \lim_{r \to 0} r^{-n} \LL^n((A_i \Delta H_i^+(x)) \cap B(x,r)) = 0
    \end{equation}
    where
    \begin{equation}
        H_i^+(x) := \set{y \in \R^n | (y - x) \cdot \nu_i(x) > 0}.
    \end{equation}
    Therefore, (\ref{eq_trace_def0}) is equivalent to
    \begin{equation}\label{eq_trace_def}
        \lim_{r \to 0} r^{-n} \int_{H_i^+(x) \cap B(x,r)} \! \abs{v(y) - v_i^*(x)} \dm\LL^n(y) = 0.
    \end{equation}
    For $\HH^{n-1}$-a.e.\ $x \in \Gamma_2$, we detail the relationship between $\overline{v}(x)^2 + \underline{v}(x)^2$ and $v_i^*(x)$.
    We fix $x \in \Gamma_2 \setminus \Gamma_1$ such that (\ref{eq_trace_def}) is satisfied for $i=2$.
    We have $v = 0$ on $A_3$ and $B(x,r)$ is disjoint from $A_1$ for small $r > 0$ so
    \begin{equation}
        \lim_{r \to 0} r^{-n} \int_{(X \setminus A_2) \cap B(x,r)} \! \abs{v} \dm \LL^n =0
    \end{equation}
    which is equivalent to
    \begin{equation}\label{eq_trace_B3}
        \lim_{r \to 0} r^{-n} \int_{H_2^-(x) \cap B(x,r)} \! \abs{v} \dm\LL^n = 0.
    \end{equation}
    Combining (\ref{eq_trace_def}) for $i=2$ and (\ref{eq_trace_B3}), we deduce
    \begin{equation}
        \overline{v}(x)^2 + \underline{v}(x)^2 = v_2^*(x)^2.
    \end{equation}
    Next, we fix $x \in \Gamma_1 \cap \Gamma_2$ such that (\ref{eq_trace_def}) holds true for $i=1$ and $i=2$.
    The surfaces $\Gamma_1$ and $\Gamma_2$ have necessary the same tangent plane at $x$ and the vectors $\nu_i$ are opposed.
    Combining (\ref{eq_trace_def}) for $i=1$ and $i=2$, we deduce
    \begin{equation}
        \overline{v}^2 + \underline{v}^2 = (v_1^*)^2 + (v_2^*)^2.
    \end{equation}

    We come back to our local minimizer $u \in \SBV_{\loc}(X)$.
    We fix $\varphi \in C^1_c(B)$.
    For $\varepsilon \in \R$, we define $v \colon X \to \R$ by
    \begin{equation}
        v :=
        \begin{cases}
            u + \varepsilon \varphi & \text{in $A_2$}\\
            u                       & \text{in $X \setminus A_2$}\\
        \end{cases}
    \end{equation}
    It is clear that $\set{v \ne u} \subset \subset B$ and that $v$ is $C^1$ in $X \setminus K$.
    As $K$ is $\HH^{n-1}$ locally finite in $X$, we have $v \in \SBV_{\loc}(X)$ and $S_v \subset K$.
    Remember that $u \geq \delta$ in $A_1 \cup A_2$, and $u = 0$ in $A_3$.
    We take $\varepsilon$ small enough so that $\varepsilon \abs{\varphi}_\infty < \delta$.
    As a consequence $v > 0$ in $A_1 \cup A_2$ and $v = 0$ in $A_3$.
    Let us check the multiplicities on the discontinuity set.
    As we have seen before, $J_v \cap B \subset \Gamma_1 \cup \Gamma_2$.
    We observe that for $x \in \Gamma_2$ such that the trace $u_2^*(x)$ exists, we have
    \begin{equation}
        v_2^*(x) = u_2^*(x) + \varepsilon \varphi(x)
    \end{equation}
    and for $x \in \Gamma_1$ such the trace $u_1^*(x)$ exists, we have
    \begin{equation}
        v_1^*(x) = u_1^*(x).
    \end{equation}
    Using the previous discussion, we deduce that for $\HH^{n-1}$-a.e.\ on
    $\Gamma_2 \setminus \Gamma_1$,
    \begin{align}
        \overline{v}^2 + \underline{v}^2    &= (u_2^* + \varepsilon \varphi)^2\\
                                            &= (\underline{u}^2 + \overline{u}^2) + 2 \varepsilon \varphi u_2^* + \varepsilon^2 \abs{\varphi}^2
    \end{align}
    that for $\HH^{n-1}$-a.e.\ on $\Gamma_2 \cap \Gamma_1$,
    \begin{align}
        \overline{v}^2 + \underline{v}^2    &= (u_2^* + \varepsilon \varphi)^2 + (u_1^*)^2\\
                                            &= (\underline{u}^2 + \overline{u}^2) + 2 \varepsilon \varphi u_2^* + \varepsilon^2 \abs{\varphi}^2
    \end{align}
    and that for $\HH^{n-1}$-a.e.\ on $\Gamma_1 \setminus \Gamma_2$,
    \begin{align}
        \overline{v}^2 + \underline{v}^2    &= (u_1^*)^2\\
                                            &=\underline{u}^2 + \overline{u}^2.
    \end{align}
    Finally, it is clear that
    \begin{multline}
        \int_B \abs{\nabla v}^2 \dm \LL^n = \int_{B} \abs{\nabla u}^2 \dm \LL^n + 2 \varepsilon \int_{A_2} \langle \nabla u, \nabla \varphi \rangle \dm \LL^n\\+ \varepsilon^2 \int_{A_2} \! \abs{\nabla \varphi}^2 \dm \LL^n.
    \end{multline}
    We plug all these informations in the minimality inequality and we obtain that
    \begin{equation}
        0 \leq 2 \varepsilon \int_{A_2} \langle \nabla u, \nabla \varphi \rangle \dm \LL^n + 2 \varepsilon \int_{\Gamma_2} \! \varphi u_2^* \dm \HH^{n-1} + C(\varphi) \varepsilon^2.
    \end{equation}
    As this holds true for all small $\varepsilon$ (positive or negative), we conclude that
    \begin{equation}
        \int_{A_2} \langle \nabla u, \nabla \varphi \rangle \dm \LL^n + \int_{\Gamma_2} \! \varphi u_2^* \dm \HH^{n-1} = 0.
    \end{equation}
\end{proof}

\section{Porosity of the singular part}

The following result says that the part where $K$ is not regular has many holes in a quantified way.
It also holds true for the almost-minimizers of \cite[Definition 2.1]{Kriventsov}.
It is simpler to obtain than its Mumford-Shah counterpart (see~\cite{Rigot}) because the $\varepsilon$-regularity theorem of~\cite{Kriventsov} only requires to control the flatness.

\begin{prop}[Porosity]\label{prop_porosity}
    Let $u \in \SBV_{\loc}(X)$ be a local minimizer.
    There exists $0 < r_0 \leq 1$, $C \geq 2$ and $0 < \alpha < 1$ (all depending on $n$,$\delta$) for which the following holds true.
    For all $x \in K$ and all $0 < r \leq r_0$ such that $B(x,r) \subset X$, there exists a smaller ball $B(y, C^{-1}r) \subset B(x,r)$ in which $K$ is $C^{1,\alpha}$-regular.
\end{prop}

\begin{proof}
    The letter $C$ is a constant $\geq 1$ that depends on $n$, $\delta$.
    The letter $\alpha$ is the constant of Theorem~\ref{thm_regularity}.
    For $y \in K$ and $t > 0$ such that $\overline{B}(y,t) \subset X$, we define the $L^\infty$ flatness
    \begin{equation}
        \beta_K(y,t) := \inf_V \sup_{z \in K \cap \overline{B}(y,t)} t^{-1} \dm(z,V),
    \end{equation}
    where the infimum is taken over the affine hyperplanes $V$ of $\R^n$ passing through $y$.
    Note that in~\cite[(41.2)]{David}, the infimum is taken over all affine hyperplanes $V$ of $\R^n$ (not necessarily passing through $y$); this would decrease our $\beta$ number but no more than a factor $\frac{1}{2}$.
    Indeed, if $V$ is any hyperplane of $\R^n$ and $y'$ is the orthogonal projection of $y$ onto $V$, then the hyperplane $V - (y'-y)$ is passing through $y$ so we have we have
    \begin{align}
        \beta_K(y,t)    &\leq \sup_{z \in K \cap \overline{B}(y,t)} \dm(z,V-(y'-y))\\
                        &\leq \abs{y' - y} + \sup_{z \in K \cap \overline{B}(y,t)} \dm(z,V)\\
                        &\leq 2\sup_{z \in K \cap \overline{B}(y,t)} \dm(z,V).
    \end{align}
    We also observe that
    \begin{equation}
        \beta_{K,2}(y,t)^2 \leq t^{-(n-1)} \HH^{n-1}(K \cap \overline{B}(y,t)) \beta_K(y,t)^2
    \end{equation}
    so as soon as $t$ is small enough for the Ahlfors-regularity to hold, we have $\beta_{K,2}(y,t) \leq C \beta_K(y,t)$.

    Let $r_0$ be the minimum between the radius of Proposition~\ref{prop_af} (Ahfors-regularity) and the radius of Proposition~\ref{prop_ur} (uniform rectifiability).
    We fix $x \in K$ and $0 < r \leq r_0$ such that $B(x,r) \subset X$.
    According to Proposition~\ref{prop_ur}, there exists an Ahlfors-regular and uniformly rectifiable set $E$ such that $K \cap \frac{1}{2}B(x,r) \subset E$.
    Moreover, the constants for the Ahfors-regularity and uniform rectifiability depends on $n$, $\delta$.
    For $y \in E$ and $t > 0$, we define as before
    \begin{equation}
        \beta_E(y,t) := \inf_V \sup_{z \in E \cap \overline{B}(y,t)} t^{-1} \dm(z,V),
    \end{equation}
    where the infimum is taken on the set of all affine hyperplanes $V$ of $\R^n$ passing through $x$.
    As $E$ is Ahlfors-regular and uniformly rectifiable, the Weak Geometric Lemma~\cite[(73.13)]{David} states that for all $\varepsilon > 0$, the set
    \begin{equation}
        \set{(y,t) | y \in E,\ 0 < t < \mathrm{diam}(E),\ \beta_E(y,t) > \varepsilon}
    \end{equation}
    is a Carleson set.
    This means that for all $\varepsilon > 0$, there exists $C_0(\varepsilon) \geq 1$ (depending on $n$, $\delta$, $\varepsilon$) such that for all $y \in E$ and all $0 < t < \mathrm{diam}(E)$,
    \begin{equation}
        \int_0^t \int_{E \cap B(y,t)} \! \mathbf{1}_{\set{\beta_E(z,s) > \varepsilon}}(z) \dm\HH^{n-1}(z) \frac{\dm s}{s} \leq C_0(\varepsilon) t^{n-1}.
    \end{equation}
    We only apply this property with $y := x$. We observe that for all $z \in K \cap B(x,\tfrac{1}{4}r)$ and for all $0 < s \leq \tfrac{1}{4} r$, we have $K \cap \overline{B}(z,s) \subset E \cap \overline{B}(z,s)$ so $\beta_K(z,s) \leq \beta_E(z,s)$.
    Thus for all $0 < t  < \mathrm{diam}(K \cap \tfrac{1}{2}B(x,r))$ such that $t \leq \tfrac{1}{4} r$, we have
    \begin{equation}
        \int_0^t \int_{K \cap B(x,t)} \! \mathbf{1}_{\set{\beta_K(z,s) > \varepsilon}}(z) \dm\HH^{n-1}(z) \frac{\dm s}{s} \leq C_0(\varepsilon) t^{n-1}.
    \end{equation}
    We only apply this property with $t := \frac{1}{4} \mathrm{diam}(K \cap B(x,r))$.
    \begin{equation}\label{eq_wgl}
        \int_0^t \int_{K \cap B(x,t)} \! \mathbf{1}_{\set{\beta_K(z,s) > \varepsilon}}(z) \dm\HH^{n-1}(z) \frac{\dm s}{s} \leq C_0(\varepsilon) t^{n-1}.
    \end{equation}
    Note that $C^{-1} r \leq t \leq \tfrac{1}{4}r$, where the first inequality comes from the Ahlfors-regularity of $K$.
    We are going to deduce from (\ref{eq_wgl}) that for all $\varepsilon > 0$, there exists $C(\varepsilon) \geq 1$, a point $z \in K \cap B(x,t)$ and a radius $s$ such that $C(\varepsilon)^{-1} t \leq s \leq t$ and $\beta_K(z,s) \leq \varepsilon$.
    We proceed by contradiction for some $C(\varepsilon)$ to be precised.
    We have therefore
    \begin{align}
        \begin{split}
        &\int_0^t \int_{K \cap B(x,t)} \! \mathbf{1}_{\set{\beta_K(z,s) > \varepsilon}}(z) \dm\HH^{n-1}(z) \frac{\dm s}{s}\\
        &\geq \HH^{n-1}(K \cap B(x,t)) \int_{C(\varepsilon)^{-1} t}^t \! \frac{\dm s}{s}
        \end{split}\\
        &\geq \HH^{n-1}(K \cap B(x,t)) \ln(C(\varepsilon))\\
        &\geq C^{-1} t^{n-1} \ln(C(\varepsilon))
    \end{align}
    This contradicts (\ref{eq_wgl}) if $C(\varepsilon)$ is too big compared to $C_0(\varepsilon)$.

    We fix $\varepsilon > 0$ (to be precised soon) and we assume that we have a corresponding pair $(z,s)$ as above.
    In particular, $\beta_{K,2}(z,s) \leq C \beta_K(z,s) \leq C \varepsilon$. According to the second statement of Theorem~\ref{thm_regularity}, we can fix $\varepsilon$ (depending on $n$, $\delta$) so that if $r_0 \leq \varepsilon$, then $K$ is $C^{1,\alpha}$-regular in $B(z,C^{-1}s)$.
\end{proof}

\section{Higher integrability of the gradient}

\begin{thm}\label{thm_integrability}
    Let $u \in \SBV(X)$ be minimal.
    There exists $0 < r_0 \leq 1$, $C \geq 1$ and $p > 1$ (depending on $n$, $\delta$) such that the following holds true.
    For all $x \in X$, for all $0 \leq r \leq r_0$ such that $B(x,r) \subset X$,
    \begin{equation}
        \int_{\tfrac{1}{2}B\left(x,r\right)} \abs{\nabla u}^{2p} \dm \LL^n \leq C r^{n-p}.
    \end{equation}
\end{thm}

The higher integrability is well known for weak solutions of elliptic systems (\cite[Theorem 2.1]{Gia}).
In this case, the proof consists in combining the Caccioppoli-Leray inequality and the Sobolev-Poincaré inequality to deduce that $\abs{\nabla u}^\frac{2n}{n+2}$ satisfies a reverse Hölder inequality.
The higher integrability is then an immediate consequence of Gehring Lemma.
In our case, $u$ is still a weak solution of an elliptic system but we lack information about the regularity of $K$ to carry out this method.

We draw inspiration from~\cite{DPF} but we simplify the proof by singling out an higher integrability lemma (Lemma~\ref{lem_technical} below) and a covering lemma (Lemma~\ref{lem_covering} below) and by removing the need of~\cite[Lemma 3.2]{DPF} (the existence of good radii).

\begin{proof}[Proof of Theorem~\ref{thm_integrability}]
    There exists $0 < r_0 \leq 1$, such that for all $x \in X$ and all $0 \leq R \leq r_0$ such that $\overline{B}(x,R) \subset X$, one can applies Lemma~\ref{lem_technical} below in the ball $B(x,R)$ to the function $v := R \abs{\nabla u}^2$.
    The assumption (i) follows from the Ahlfors-regularity of $K$ (Proposition~\ref{prop_af}).
    The assumption (ii) follows from the porosity (Proposition~\ref{prop_porosity}).
    The assumption (iii) follows from interior/boundary gradient estimates for the Robin problem and from the Ahlfors-regularity.
    In particular, the interior estimate can be derived from the subharmonicity of $\abs{\nabla u}^2$ in $X \setminus K$ and the boundary estimate is detailed in Lemma~\ref{lem_robin_estimate} in Appendix~\ref{appendix_robin}.
\end{proof}

\begin{lem}\label{lem_technical}
    We fix a radius $R > 0$ and an open ball $B_R$ of radius $R$.
    Let $K$ be a closed subset of $B_R$ and $v \colon B_R \to \R^+$ be a non-negative Borel function.
    We assume that there exists $C_0 \geq 1$ and $0 < \alpha \leq 1$ such that the following holds true.
    \begin{enumerate}[label = (\roman*)]
        \item
            For all ball $B(x,r) \subset B_R$,
            \begin{equation}
                C_0 r^{n-1} \leq \HH^{n-1}(K \cap B(x,r)) \leq C_0 r^{n-1}.
            \end{equation}

        \item
            For all ball $B(x,r) \subset B_R$ centered in $K$, there exists a smaller ball $B(y,C_0^{-1} r) \subset B(x,r)$ in which $K$ is $C^{1,\alpha}$-regular (Definition~\ref{defi_regular}).

        \item
            For all ball $B(x,r) \subset B_R$ such that $K$ is disjoint from $B(x,r)$ or $K$ is $C^{1,\alpha}$-regular in $B(x,r)$ (Definition~\ref{defi_regular}), we have
            \begin{equation}
                \sup_{\tfrac{1}{2}B(x,r)} v(x) \leq C_0 \left(\frac{R}{r}\right).
            \end{equation}
    \end{enumerate}
    Then there exists $p > 1$ and $C \geq 1$ (depending on $n$, $C_0$) such that
    \begin{equation}
        \fint_{\tfrac{1}{2}B_R} v^p \leq C.
    \end{equation}
\end{lem}

The proof of Lemma~\ref{lem_technical} takes advantage of the following covering lemma.
We use the notation $\Gamma_{(x,e_n)}$ defined at line (\ref{eq_gamma}).
The assumption (ii) says that in each double ball $2 B_k$, the set $E$ is an union of Lipschitz graphs which are close to an hyperplane.

\begin{lem}[Covering Lemma]\label{lem_covering}
    Let $E \subset \R^n$ be a bounded set.
    Let $(B_k)$ be a family of open balls of center $x_k \in \R^n$ and radius $R_k > 0$.
    We assume that
    \begin{enumerate}[label = (\roman*)]
        \item
            for all $k \ne l$, $2B_k \cap B_l = \emptyset$;

        \item
            for all $k$, for all $x \in E \cap 2 B_k$, there exists a vector $e_n \in \mathbf{S}^{n-1}$ and a $\frac{1}{2}$-Lipschitz function $f \colon e_n^\perp \to \R$ such that $\abs{f} \leq \frac{1}{2} R_k$ and
            \begin{equation}
                x \in \Gamma_{(x_k,e_n)}(f) \cap 2 B_k \subset E.
            \end{equation}
    \end{enumerate}
    Let $0 < r \leq \inf_k R_k$.
    There exists a sequence of open balls $(D_i)_{i \in I}$ of radius $r$ and centered in $E \setminus \bigcup_k B_k$ such that
    \begin{equation}
        E \setminus \bigcup_k B_k \subset \bigcup_{i \in I} D_i
    \end{equation}
    and the balls $(20^{-1} D_i)_{i \in I}$ are pairwise disjoint and disjoint from $\bigcup_k B_k$.
\end{lem}

\begin{proof}
    Let $0 < r_0 \leq \inf_k R_k$.
    We introduce the set
    \begin{equation}
        F := E \setminus \bigcup_k B_k.
    \end{equation}
    The goal is to cover $F$ with a controlled number of balls of radius $r_0$.
    Let $r$ be a radius $0 < r \leq r_0$ which will be precised during the proof.
    As $F$ is bounded, there exists a maximal sequence of points $(x_i) \in F$ such that $B(x_i,r) \subset \R^n \setminus \bigcup_k B_k$ and $\abs{x_i - x_j} \geq r$.
    For $i \ne j$, we have $\abs{x_i - x_j} \geq r$ so the balls $(B(x_i, \tfrac{1}{2}r))_i$ are disjoint.
    Next, we show that
    \begin{equation}
        F \subset \bigcup_i B(x_i, 10 r).
    \end{equation}
    Let $x \in F$.
    If $B(x,r) \subset \R^n \setminus \bigcup_k B_k$, then by maximality of $(x_i)$, there exists $i$ such that $x \in B(x_i,r) \subset B(x_i,10r)$.
    Now we focus on the case where there exists an index $k_0$ such that $B(x,r) \cap B_{k_0} \ne \emptyset$.
    The radius of $B_{k_0}$ is denoted by $R$ and we assume without loss of generality that its center is $0$.
    As $x \in F = E \setminus \bigcup_k B_k$ and $B(x,r) \cap B(0,R) \ne \emptyset$, we have $R < \abs{x} < R + r$.
    We are going to build a point $y \in E$ such that $R + r < \abs{y} < R + 7r$ and $\abs{x - y} < 9r$.

    Since $r \leq R$, we observe that $x \in B(0,2R)$.
    According to the assumptions of the lemma, there exists two scalars $0 < \varepsilon, L \leq \frac{1}{2}$, a vector $e_n \in \mathbf{S}^{n-1}$ and a $L$-Lipschitz $f \colon e_n^\perp \to R$ such that $\abs{f} \leq \varepsilon R$ and
    \begin{equation}
        x \in \set{y \in B(0,2R) | y_n = f(y')} \subset E.
    \end{equation}
    Here, we have decomposed each point $y \in \R^n$ under the form $y = y' + y_n e_n$ where $y' \in n^\perp$ and $y_n \in \R$.
    The estimate $R < \abs{x} < R + r$ can be rewritten
    \begin{equation}
        R < \abs{x' + f(x')e_n} < R + r.
    \end{equation}
    We consider $t \geq 1$ such that $\abs{t x' + f(x')e_n} = R + 4r$ and we estimate how close $t x'$ is to $x'$.
    We have
    \begin{align}
        \abs{x'}    &\geq \sqrt{R^2 - \abs{f(x')}^2}\\
        \abs{t x'}  &\leq \sqrt{(R + 4r)^2 - \abs{f(x')}^2}
    \end{align}
    so
    \begin{align}
        \abs{t x' - x'} &\leq \sqrt{(R + 4r)^2 - \abs{f(x')}^2} - \sqrt{R^2 - \abs{f(x')}^2}\\
                        &\leq \frac{4 R r + 8 r^2}{\sqrt{R^2 - \abs{f(x')}^2}}.
    \end{align}
    We assume $r \leq \tfrac{1}{8}R$ and we recall that $\abs{f(x')} \leq \varepsilon R$ with $\varepsilon \leq \frac{1}{2}$ so this simplifies to
    \begin{equation}
        \abs{t x' - x'} \leq \frac{5 r}{\sqrt{1 - \varepsilon^2}} < 6r.
    \end{equation}
    Next, we define $y := tx' + f(tx')e_n$ and we recall that $f$ is $L$-Lipschitz with $L \leq \frac{1}{2}$ to estimate
    \begin{align}
        \abs{y - [t x' + f(x')e_n]} &= \abs{f(t x') - f(x')}\\
                                    &< 3 r.
    \end{align}
    Since $\abs{t x' + f(x')e_n} = R + 4 r$, this yields
    \begin{equation}\label{eq_y1}
        R + r < \abs{y} < R + 7 r.
    \end{equation}
    We also estimate
    \begin{align}
        \abs{y - x} &\leq \abs{t x' - x'} + \abs{f(tx') - f(x')}\\
                    &< 9r\label{eq_y2}.
    \end{align}
    As $r \leq \tfrac{1}{8}R$, the inequalities (\ref{eq_y1}) imply $y \in B(0,2R)$ and thus $y \in E$.
    We are going to justify that $B(y,r) \subset \R^n \setminus \bigcup_k B_k$.
    We recall that $B(0,2R)$ is disjoint from all the other balls of the family $(B_k)$.
    By (\ref{eq_y1}), we observe that
    \begin{align}
        B(y,r)  &\subset B(0,R + 8r) \setminus B(0,R)\\
                &\subset B(0,2R) \setminus B(0,R)
    \end{align}
    and our claim follows.
    By maximality of the family $(x_i)$, there exists $i$ such that $\abs{y - x_i} < r$ and in turn by (\ref{eq_y2}), $\abs{x - x_i}
    < 10 r$.
    We finally choose $r := \tfrac{1}{10} r_0$.
    The balls $(D_i)$ are given by $D_i := B(x_i,10r) = B(x_i,r_0)$.
\end{proof}

\begin{proof}[Proof of Lemma~\ref{lem_technical}]
    We observe that any ball $B \subset B_R$ and for $p \geq 1$,
    \begin{align}
        \int_{B} \! v^p \dm\LL^n    &= \int_0^\infty \LL^n(B \cap \set{v^p > t }) \dm t\\
                                    &= p \int_0^\infty s^{p - 1} \LL^n(B \cap \set{v > s}) \dm s
    \end{align}
    and for $M \geq 1$,
    \begin{align}
        \begin{split}
        &p \int_1^\infty s^{p - 1} \LL^n(B \cap \set{v > s}) \dm s\\
        &\qquad \leq p \sum_{h = 0}^\infty \int_{M^h}^{M^{h+1}} s^{p - 1} \LL^n(B \cap \set{v > s}) \dm s
        \end{split}\\
        &\qquad \leq p \sum_{h = 0}^\infty \left(\int_{M^h}^{M^{h+1}} s^{p - 1} \dm s\right) \LL^n(B \cap \set{v > M^h})\\
        &\qquad \leq (M^{p} - 1) \sum_{h = h}^\infty M^{hp} \LL^n(B \cap \set{v > M^h}).
    \end{align}
    Thus, it suffices to prove that there exists $N > M \geq 1$, $C \geq 1$ (depending on $n$, $C_0$) such that for all $h \geq 0$,
    \begin{equation}
        \LL^n(\tfrac{1}{2}B_R \cap \set{v > M^h}) \leq C R^n N^{-h}
    \end{equation}
    and then take $p > 1$ such that $M^p N^{-1} < 1$.

    To simplify the notations, we change the constant $C_0$ so that (ii) yields that $40 B(y,C_0^{-1} r) \subset B(x,r)$ and that $K$ is $C^{1,\alpha}$-regular in $4 B(y, C_0^{-1}r)$.

    Let $M := \max\set{4 C_0, \tfrac{1}{4} C_0^2} \geq 4$.
    We define for $h \geq 1$,
    \begin{equation}
        A_h := \set{x \in \tfrac{1}{2}B_R \setminus K | v > M^h}.
    \end{equation}
    The proof is based on the fact that $A_h$ is at distance $\sim M^{-h} R$ from $K$ and has many holes of size $\sim M^{-h} R$ near $K$.
    We justify more precisely these observations.
    For the first one, let $h \geq 1$, let $x \in A_h$ and assume that $B(x, C_0 M^{-h} R)$ is disjoint from $K$.
    Then we use property (ii) to estimate
    \begin{equation}
        v(x) \leq M^h.
    \end{equation}
    This contradicts the definition of $A_h$.
    We deduce that there exists $y \in K$ such that $\abs{x - y} < C_0 M^{-h} R$.
    For the second observation, let $h \geq 2$, let $x \in K \cap \tfrac{15}{16} B_R$ and apply the porosity property to the ball $B(x, M^{-h} R)$.

    We obtain an open ball $B \subset X$ centered in $K$, of radius $C_0^{-1} M^{-h} R$ and such that $K$ is $C^{1,\alpha}$-regular in $4B$.
    Then by point (iii) and since $M \geq \tfrac{1}{4} C_0^2$,
    \begin{equation}
        \sup_{2B} v \leq \tfrac{1}{4} C_0^2 M^h \leq M^{h+1}
    \end{equation}
    In particular, $2B$ is disjoint from $A_{h+1}$.

    We start the proof by defining for $h \geq 1$,
    \begin{align}
        r(h)     &:= M^{-h}R\\
        R(h)     &:= \left(\frac{3}{4} + M^{-h+1}\right) R.
    \end{align}
    The sequence $(R(h))$ is decreasing, $\lim_{h \to \infty} R(h) = \tfrac{3}{4}R$ and $R(h+1) + r(h) \leq R(h)$.
    For each $h \geq 1$, we build an index set $I(h)$ and a family of balls $(B_i)_{i \in I(h)}$ as follow.
    First we define $I(1) := \emptyset$ and $(B_i)_{i \in I(1)} := \emptyset$.
    Let $h \geq 2$ be such that $(B_i)_{i \in I(1)}, \ldots, (B_i)_{i \in I(h-1)}$ have been built.
    We assume that the index sets $I(g)$, where $g=1,\ldots,h-1$, are pairwise disjoint.
    We assume that for all $i \in I_g$, the balls $B_i$ have radius $C_0^{-1} r(g) = C_0^{-1} M^{-g} R$.
    We assume that for all indices $i, j \in \bigcup_{g=1}^{h-1} I(g)$ with $i \ne j$, we have that $2 B_i \cap B_j = \emptyset$ and that $K$ is $C^{1,\alpha}$-regular in $2 B_i$.
    Then, we introduce the sets
    \begin{align}
        K_h     &:= K \cap B_{R(h)} \setminus \bigcup_{g=1}^{h-1} \bigcup_{i \in I(g)} B_i\\
        K_h^*   &:= K \cap B_{R(h+1)} \setminus \bigcup_{g=1}^{h-1} \bigcup_{i \in I(g)} B_i.
    \end{align}
    According to Lemma~\ref{lem_covering}, there exists a sequence of open balls $(D_i)_{i \in I(h)}$ centered in $K_h^*$ of radius $r(h) = M^{-h} R$ such that
    \begin{equation}
        K_h^* \subset \bigcup_{i \in I(h)} D_i,
    \end{equation}
    and such that the balls $(20^{-1} D_i)$ are pairwise disjoint and disjoint from $\bigcup_{g=1}^{h-1} \bigcup_{i \in I(g)} B_i$.
    We can assume that index set $I(h)$ is disjoint from the sets $I(g)$, $g=1,\ldots,h-1$.
    Since $R(h+1) + r(h) \leq R(h)$, we observe that the balls $(20^{-1} D_i)$ are included in
    \begin{equation}
        B_{R(h)} \setminus \bigcup_{g=1}^{h-1} \bigcup_{i \in I(g)} B_i.
    \end{equation}
    Next, we apply the porosity to the balls $(D_i)$.
    For each $i \in I(h)$, there exists an open ball $B_i$ centered in $K$, of radius $C_0^{-1} M^{-h} R$ such that $B_i \subset 40^{-1}D_i$, $K$ is $C^{1,\alpha}$-regular in $4 B_i$ and by (iii),
    \begin{equation}
        \sup_{2B_i} v \leq \tfrac{1}{4} C_0^2 M^h \leq M^{h+1}
    \end{equation}
    We should not forget to mention that for all $i \in I(h)$, we have $2B_i \subset 20^{-1} D_i$ so $2B_i$ is disjoint from all the other balls we have built so far.

    Now, we estimate $\LL^n(A_h)$ for $h \geq 1$.
    We show first that the points of $A_h$ cannot be too far from $K_h^*$.
    Let $x \in A_h$.
    We have seen earlier that there exists $y \in K$ such that $\abs{x - y} < C_0 M^{-h} R$.
    We are going to show that $y \in K_h^*$.
    Since $\abs{x} \leq \tfrac{1}{2}R$ and $M \geq 4C_0$, we have
    \begin{equation}
        \abs{y} \leq \tfrac{1}{2}R + C_0 M^{-h}R \leq \tfrac{3}{4}R.
    \end{equation}
    Let us assume that there exists $g=1,\ldots,h-1$ and $i \in I(g)$ such that $y \in B_i$.
    The radius of $B_i$ is $C_0^{-1} M^{-(h-1)} R$ and since $\abs{x - y} < C_0 M^{-h} R$, we have $x \in 2B_i$.
    However $2B_i$ is disjoint from $A_h$ by construction.
    We have shown that $y \in K_h^*$.
    As a consequence, there exists $i \in I(h)$ such that $y \in D_i$.
    The radius of $D_i$ is $r(h) = M^{-h} R$ and $\abs{x - y} < C_0 M^{-h} R$ so
    \begin{equation}
        A_h \subset \bigcup_{i \in I(h)} (1 + C_0)D_i.
    \end{equation}
    This allows to estimate
    \begin{equation}
        \LL^n(A_h) \leq \omega_n (1 + C_0)^n \abs{I(h)} r(h)^n
    \end{equation}
    where $\omega_n$ is the Lebesgue measure of the unit ball.

    Next, we want to control $\abs{I(h)}$.
    The balls $(20^{-1} D_i)_{i \in I(h)}$ are disjoint and included in the set $B(R(h)) \setminus \bigcup_{g=2}^{h-1} \bigcup_{i \in I(g)} B_i$ so by Ahlfors-regularity,
    \begin{align}
        C_0^{-1} 20^{-(n-1)} r(h)^{(n-1)} \abs{I(h)} &\leq \sum_{i \in I(h)} \HH^{n-1}(K \cap 12^{-1} D_i)\\
                                                  &\leq \HH^{n-1}(K_h).
    \end{align}

    We are going to see that $\HH^{n-1}(K_h)$ is bounded from above by a decreasing geometric sequence.
    We have
    \begin{align}
        \HH^{n-1}(K_h^*)    &\leq \sum_{i \in I(h)} \HH^{n-1}(K \cap D_i)\\
                            &\leq C_0 \sum_{i \in I(h)} r(h)^{n-1}\\
                            &\leq C_0^{n+1} \sum_{i \in I(h)} (C_0^{-1} r(h))^{n-1}\\
                            &\leq C_0^{n+1} \sum_{i \in I(h)} \HH^{n-1}(K \cap B_i)\\
                            &\leq C_0^{n+1} \HH^{n-1}(K_h \setminus K_{h+1}).
    \end{align}
    We deduce
    \begin{equation}
    \HH^{n-1}(K_h) \leq C_0^{n+1} \HH^{n-1}(K_h \setminus K_{h+1}) + \HH^{n-1}(K \cap B_{R(h)} \setminus B_{R(h+1)}).
    \end{equation}
    We rewrite this inequality as
    \begin{equation}
        \HH^{n-1}(K_{h+1}) \leq \lambda^{-1} \HH^{n-1}(K_h) + C_0^{-(n+1)} \HH^{n-1}(K \cap B_{R(h)} \setminus B_{R(h+1)})
    \end{equation}
    where $\lambda := C_0^{n+1}(C_0^{n+1} - 1)^{-1} > 1$.
    Then, we multiply both sides of the inequality by $\lambda^{h+1}$:
    \begin{align}
        \begin{split}
            &\lambda^{h+1} \HH^{n-1}(K_{h+1})\\
            &\qquad \leq \lambda^h \HH^{n-1}(K_h) + C_0^{-(n+1)} \lambda^{-h} \HH^{n-1}(K \cap B_{R(h)} \setminus B_{R(h+1)})
        \end{split}\\
            &\qquad \leq \lambda^h \HH^{n-1}(K_h) + \HH^{n-1}(K \cap B_{R(h)} \setminus B_{R(h+1)}).
    \end{align}
    Summing this telescopic inequality, we obtain that for all $h \geq 1$,
    \begin{align}
        \lambda^h \HH^{n-1}(K_h)    &\leq 2\HH^{n-1}(K \cap B_R)\\
                                    &\leq 2 C_0 R^{n-1}.
    \end{align}
    In summary, we have proved that for some constant $C \geq 1$, $\lambda > 1$ (depending on $n$, $C_0$) and for $h \geq 1$
    \begin{equation}
        \LL^n(A_h) \leq C R^n (\lambda M)^{-h}.
    \end{equation}
\end{proof}

\section{Dimension of the singular part}
\textbf{Notation}. The Hausdorff dimension of a set $A \subset \R^n$ is defined by
\begin{equation}
    \dim_{\HH}(A) := \inf \set{s \geq 0 | H^s(A) = 0}.
\end{equation}
We take the convention that for $s < 0$, the term \emph{$\HH^s$-a-e.} means \emph{everywhere} and the inequality \emph{$\dim_{\HH}(A) < 0$} means \emph{$A = \emptyset$}.

The goal of this section is to explain the link between the integrability exponent of the gradient and the dimension of the singular part.
It has been first observed for the Mumford-Shah functional by Ambrosio, Fusco, Hutchinson in~\cite{AFH}.

\begin{thm}\label{thm_dimension}
    Let $u \in \SBV_{\loc}(X)$ be a local minimizer.
    We define
    \begin{equation}
        \Sigma := \set{x \in K | \text{$K$ is not regular at $x$}}.
    \end{equation}
    For $p > 1$ such that $\abs{\nabla u}^2 \in L^p_{\loc}(X)$, we have
    \begin{equation}
        \dim_{\HH}(\Sigma) \leq \max \set{n-p, n-8} < n-1.
    \end{equation}
\end{thm}

\begin{rmk}\label{rmk_dimension}
    In dimension $n \leq 7$, Caffarelli--Kriventsov have shown that if a point $x \in K$ is at the boundary of two local connected components where $u > 0$ or if it is a $0$-density point of $\set{u=0}$, then $x$ is a regular point (\cite[Theorem 8.2]{CK}). In dimension $n=2$, they show furthermore that if $x$ is at the boundary of a connected component of $\set{u=0}$, then it is a regular point (\cite[Corollary 9.2]{CK}). Thus in the planar case, a point of $\Sigma$ must be an acculumation point of connected components of $\set{u=0}$. There is however no known example of such a situation.
\end{rmk}

Theorem~\ref{thm_dimension} will be proved very easily with the help of \cite[Theorem 8.2]{CK} and the following well-known result.

\begin{lem}\label{lem_int}
    Let $v \in L^p_{\loc}(X)$ for some $p \geq 1$ and let $s < n$.
    Then, for $\HH^{n - p (n - s)}$-a.e.\ $x \in X$,
    \begin{equation}
        \lim_{r \to 0} r^{-s} \int_{B(x,r)} \! v \dm\LL^n = 0.
    \end{equation}
\end{lem}

\begin{proof}
    Without loss of generality, we assume $v \geq 0$.
    We start with the case $p=1$.
    We define $\mu$ as the measure $v \LL^n$ and we want to show that for $\HH^s$-a.e.\ $x \in X$, we have
    \begin{equation}
        \lim_{r \to 0} r^{-s} \mu(B(x,r)) = 0.
    \end{equation}
    If $s < 0$, the limit is indeed $0$ for every $x \in X$.
    In the case $0 \leq s < n$, we fix a closed ball $\overline{B} \subset X$, a scalar $\lambda > 0$ and a set
    \begin{equation}
        A := \set{x \in \overline{B} | \limsup_{r \to 0} r^{-s} \mu(B(x,r)) > \lambda}.
    \end{equation}
    According to~\cite[Theorem 2.56]{Ambrosio},
    \begin{equation}\label{eq_muA}
        \mu(A) \geq \lambda H^s(A).
    \end{equation}
    As $A \subset \overline{B}$ and $\mu$ is a Radon measure, we have $\mu(A) < \infty$.
    Then (\ref{eq_muA}) gives $H^s(A) < \infty$ and since $s < n$, $\LL^n(A) = 0$.
    The measure $\mu$ is dominated by $\LL^n$ so $\mu(A) = 0$ and now (\ref{eq_muA}) gives $H^s(A) = 0$.
    We can take a sequence of scalars $\lambda_k \to 0$ to deduce
    \begin{equation}
        \HH^s(\set{x \in \overline{B} | \limsup_{r \to 0} r^{-s} \mu(B(x,r)) > 0}) = 0.
    \end{equation}
    We can then conclude that
    \begin{equation}
        \HH^s(\set{x \in X | \limsup_{r \to 0} r^{-s} \mu(B(x,r)) > 0}) = 0.
    \end{equation}
    by covering $X$ with a a sequence of closed balls $\overline{B_k} \subset X$.

    Now we come to the general case $p \geq 1$.
    Let us fix $t < n$.
    For $x \in X$ and for $r > 0$, the Hölder inequality shows that
    \begin{equation}
        r^{-\left(n - \frac{n}{p}\right)} \int_{B(x,r)} \! v \dm\LL^n \leq \left(\int_{B(x,r)} \! v^p \dm\LL^n\right)^{\frac{1}{p}}
    \end{equation}
    so
    \begin{equation}
        r^{-\left(n + \frac{t}{p} - \frac{n}{p}\right)} \int_{B(x,r)} \! v \dm\LL^n \leq \left(r^{-t} \int_{B(x,r)} \! v^p \dm\LL^n\right)^{\frac{1}{p}}.
    \end{equation}
    We apply the first part to see that for $\HH^t$-a.e.\ $x \in X$,
    \begin{equation}
        \lim_{r \to 0} r^{-\left(n + \frac{t}{p} - \frac{n}{p}\right)} \int_{B(x,r)} \! v \dm\LL^n = 0.
    \end{equation}
    The scalar $t$ such that $s = n + \frac{t}{p} - \frac{n}{p}$ is $t := n - p(n-s) < n$.
\end{proof}

\begin{proof}[Proof of Theorem~\ref{thm_dimension}]
    According to Lemma~\ref{lem_int}, we have for $\HH^{n-p}$-a.e.\ $x \in X$,
    \begin{equation}
        \lim_{r \to 0} \omega_2(x,r) = 0
    \end{equation}
    and according to~\cite[Theorem 8.2]{CK}, the set
    \begin{equation}
        \Set{x \in X \cap \Sigma | \lim_{r \to 0} \omega_2(x,r) = 0}
    \end{equation}
    has a Hausdorff dimension $\leq n-8$.
\end{proof}

\begin{appendices}\label{appendix}

    \section{Generalities about \texorpdfstring{$\BV$} functions}\label{appendix_bv}

    We recall a few definitions and results from the theory of $\BV$ functions (\cite{Ambrosio}).
    We work in an open set $X$ of the Euclidean space $\R^n$ ($n > 1$).
    When a point $x \in X$ is given, we abbreviate the open ball $B(x,r)$ as $B_r$.

    Let $u \in L^1_{loc}(X)$.
    The \emph{upper and lower approximate limit} of $u$ at at a point $x \in X$ are defined by
    \begin{align}
        \overline{u}(x)     &:= \inf \set{t \in \R | \lim_{r \to 0} r^{-n} \int_{\set{u > t} \cap B_r} \! (u - t) \dm\LL^n = 0},\\
        \underline{u}(x)    &:= \sup \set{t \in \R | \lim_{r \to 0} r^{-n} \int_{\set{u < t} \cap B_r} \! (t - u) \dm\LL^n = 0}.
    \end{align}
    The functions $\overline{u}, \underline{u}\colon X \to \overline{\R}$ are Borel and satisfies $\underline{u} \leq \overline{u}$.
    We have two examples in mind.
    We say that $x$ is a \emph{Lebesgue point} if there exists $t \in \R$ such that
    \begin{equation}
        \lim_{r \to 0} \fint_{B_r} \! \abs{u - t} \dm\LL^n = 0.
    \end{equation}
    We then have $\underline{u}(x) = \overline{u}(x) = t$ and we denote $t$ by $\tilde{u}(x)$.
    The set of non-Lebesgue points $x \in X$ is called the \emph{singular set} $S_u$.
    Both the set $S_u$ and the function $X \setminus S_u \to \R, x \mapsto \tilde{u}(x)$ are Borel (\cite[Proposition 3.64]{Ambrosio}).
    The Lebesgue differentiation theorem states that for $\LL^n$-a.e.\ $x \in X$, we have $x \in X \setminus S_u$ and $u(x) = \tilde{u}(x)$.
    We say that $x$ is a \emph{jump point} if there exist two real numbers $s < t$ and a (unique) vector $\nu_u(x) \in \mathbf{S}^{n-1}$ such that
    \begin{subequations}
        \begin{align}
    &\lim_{r \to 0} \fint_{H^+ \cap B_r} \! \abs{u(y) - s} \dm\LL^n(y) = 0\\
    &\lim_{r \to 0} \fint_{H^- \cap B_r} \! \abs{u(y) - t} \dm\LL^n(y) = 0,
        \end{align}
    \end{subequations}
    where
    \begin{subequations}
        \begin{align}
            H^+ &:= \set{y \in \R^n | (y - x) \cdot \nu_u(x) > 0}\\
            H^- &:= \set{y \in \R^n | (y - x) \cdot \nu_u(x) < 0}.
        \end{align}
    \end{subequations}
    We then have $\overline{u}(x) = t$ and $\underline{u}(x) = s$.
    The set of jump points $x \in X$ is called the \emph{jump set} $J_u$.
    Both the set $J_u$ and the function $J_u \to \mathbf{S}^{n-1}, x \mapsto \nu_u(x)$ are Borel (\cite[Proposition 3.69]{Ambrosio}).

    Here we summarize~\cite[Proposition 3.76, 3.78]{Ambrosio}.
    Let $u \in \BV(X)$.
    The singular set $S_u$ is $\HH^{n-1}$ rectifiable and $\HH^{n-1}(S_u \setminus J_u) = 0$.
    According to the Besicovitch derivation theorem, we can write
    \begin{equation}
        Du = D^a u + D^s u
    \end{equation}
    where $D^a u$ is the \emph{absolutely continuous} part of $Du$ with respect to $\LL^n$ and $D^s u$ is the \emph{singular part} of $Du$ with respect to $\LL^n$.
    As a consequence, there exists a unique vector-valued map $\nabla u \in L^1(X;R^n)$, the \emph{approximate gradient}, such that $D^a u = \nabla u \LL^n$.
    The measures $\LL^n$ and $\norm{D^s u}$ are mutually singular which means that there exists a Borel set $S \subset X$ such that
    \begin{equation}\label{eq_mutually_singular}
        \LL^n(S) = \norm{D^s u}(\R^n \setminus S) = 0.
    \end{equation}
    A candidate for $S$ could be $S_u$ but $S$ may not be a $(n-1)$ dimensional set.
    We can write
    \begin{equation}
        D^s u = Du^s \mres S_u + D^s u \mres (X \setminus S_u)
    \end{equation}
    where $D^s u \mres S_u$ is the \emph{jump part} and $D^s u \mres (X \setminus S_u)$ is the \emph{Cantor part}.
    The jump part has an explicit formula,
    \begin{equation}
        D^s u \mres S_u = (\overline{u}-\underline{u}) \nu_u \HH^{n-1} \mres J_u,
    \end{equation}
    whereas the Cantor part vanishes on $\HH^{n-1}$ $\sigma$-finite sets $B \subset X$ (and not only $B := S_u$).
    Remark that $Du$ always vanishes on $\HH^{n-1}$ negligible sets.
    Finally, we define $\SBV(X)$ as the subspace of functions $u \in \BV(X)$ whose Cantor part is zero, that is
    \begin{equation}
        \norm{D^s u}(X \setminus S_u) = 0.
    \end{equation}

    For $u \in \BV(X)$ to be $\SBV$, it suffices that there exists a $\HH^{n-1}$ $\sigma$-finite set $K \subset X$ such that $\norm{D^s u} (X \setminus K) = 0$.
    Since $\norm{D^s u}$ and $\LL^n$ are mutually singular, this also amounts to say the measure $\norm{Du} \mres (X \setminus K)$ is dominated by $\LL^n$.
    A natural way to build $\SBV(X)$ functions is to have a pair $(u,K)$ where $K \subset X$ is relatively closed, $\HH^{n-1}$ locally finite and $u \in W^{1,1}(X \setminus K)$.

    A set of finite perimeter in $X$ is a Borel set $E \subset X$ such that $\mathbf{1}_E \in \BV(X)$.
    The singular set of $\mathbf{1}_E$ is called \emph{essential boundary} or \emph{measure-theoretic boundary} and denoted by $\partial_M E$.
    The jump set of $E$ is denoted by $\partial^* E$.
    One can see that $\overline{\mathbf{1}_E}, \underline{\mathbf{1}_E} \in \set{0,1}$ everywhere on $X$.
    Thus, if $x$ is a Lebesgue point of $\mathbf{1}_E$, we have either $\lim_{r \to 0} r^{-n} \LL^n((X \setminus E) \cap B_r) = 0$ or $\lim_{r \to 0} r^{-n} \LL^n(E \cap B_r) = 0$.
    The essential boundary $\partial_M E$ can be reformulated as the set of points $x \in X$ such that
    \begin{subequations}
        \begin{align}
            &\limsup_{r \to 0} r^{-n} \LL^n((X \setminus E) \cap B_r) > 0\\
            &\limsup_{r \to 0} r^{-n} \LL^n(E \cap B_r) > 0.
        \end{align}
    \end{subequations}
    Similarly, the jump set $\partial^* E$ can be reformulated as the set of points $x \in X$ for which there exists a (unique) vector $n_E(x) \in \mathbf{S}^{n-1}$ such that
    \begin{equation}
        \lim_{r \to 0} r^{-n} \LL^n((E \Delta H^+) \cap B_r) = 0
    \end{equation}
    where
    \begin{equation}
        H^+ := \set{y \in \R^n | (y - x) \cdot n_E(x) > 0}.
    \end{equation}
    The vector $n_E(x)$ is called the \emph{measure-theoretic inner normal} to $E$ at $x$.
    We have the inclusions $\partial^* E \subset \partial_M E \subset \partial E$.
    The measure $D\mathbf{1}_E$ has no absolutely continuous part, neither Cantor part; it is given by the formula
    \begin{equation}
        D\mathbf{1}_E = n_E \HH^{n-1} \mres \partial^* E.
    \end{equation}

    \section{A Robin problem}\label{appendix_robin}

    \subsection{Statement}

    We work in the Euclidean space $\R^n$ ($n > 1$). For $r > 0$, $B_r$ denotes the ball of radius $r$ and centered at $0$. We fix a radius $0 < R \leq 1$, an exponent $0 < \alpha \leq 1$, a constant $A > 0$ and a $C^{1,\alpha}$ function $f\colon \R^{n-1} \cap B_R \to \R$ such that $f(0) = 0$, $\nabla f(0) = 0$ and $R^\alpha \left[\nabla f\right]_\alpha \leq A$. We introduce
    \begin{align}
        V_R         &:= \set{x \in B_R | x_n > f(x')}\\
        \Gamma_R    &:= \set{x \in B_R | x_n = f(x')}.
    \end{align}
    We denote by $\nu$ the normal vector field to $\Gamma_R$ going upward.
    For $0 < t \leq 1$, we write $t V_R$ for $V_R \cap B_t$ and $t \Gamma_R$ for $\Gamma_R \cap B_t$.
    For $u \in \BV(V_R)$, we denote by $u^*$ the trace of $u$ in $L^1(\partial V_R)$.
    It is characterized by the property that for $\HH^{n-1}$-a.e.\ $x \in \partial V_R$,
    \begin{equation}\label{eq_char_trace}
        \lim_{r \to 0} r^{-n} \int_{V_R \cap B(x,r)} \abs{u - u^*(x)} \dm\LL^n = 0.
    \end{equation}
    We denote by $W^{1,2}_0(V_R \cup \Gamma_R)$ the space of functions $u \in W^{1,2}(V_R)$ such that $u^* = 0$ on $\partial V_R \setminus \Gamma_R$.
    Our object of study are the functions $u \in W^{1,2}(V_R) \cap L^\infty(V_R)$ which are weak solutions of
    \begin{equation}\label{eq_robin_problem}
        \left\{
            \begin{array}{lcl}
                \Delta u            & = & 0 \quad \text{in $V_R$}\\
                \partial_{\nu}u - u & = & 0 \quad \text{in $\Gamma_R$},
            \end{array}
        \right.
    \end{equation}
    that is for all $v \in W^{1,2}_0(V_R \cup \Gamma_R)$,
    \begin{equation}
        \int_{V_R} \! \langle \nabla u, \nabla v \rangle \dm\LL^n + \int_{\Gamma_R} \! u^*v^* \dm\HH^{n-1} = 0.
    \end{equation}

    According to Weyl's lemma, $u$ coincide almost-everywhere in $V_R$ with an harmonic functions.
    We replace $u$ by this harmonic representative so that $u$ is pointwise defined and smooth in $V_R$.
    Our goal is to prove the following estimate.
    \begin{lem}\label{lem_robin_estimate}
        There exists $C \geq 1$ (depending on $n$, $\alpha$, $A$) such that
        \begin{equation}
            \abs{\nabla u}_\infty \leq C \left(\fint_{V_R} \abs{\nabla u}^2 \dm\LL^n\right)^{\frac{1}{2}} + C \abs{u}_\infty,
        \end{equation}
        where the left-hand side is computed on $\tfrac{1}{2} V_R$.
    \end{lem}
    Although well known to experts, we present an elementary proof because we have not found a satisfactory reference with the estimate as asserted.
    It will be an immediate consequence of Lemma~\ref{lem_robin_schauder} and Lemma~\ref{lem_robin_boundedness}.

    According to~\cite{Zhang}, the viscosity solutions of such a problem are pointwise $C^{1,\alpha}$ up to the boundary.
    The viscosity approach is based on the maximum principle but it is easy to prove a maximum principle for the weak solutions of our problem and thus to follow the ideas of~\cite{Zhang}.
    Once we get that $u$ is $C^\alpha$ up to the boundary, the Robin boundary condition can be written as a Neumann boundary condition with a $C^{\alpha}$ right-hand side and we will apply the usual estimates for Neumann problems.

    \begin{lem}[Maximum principle]\label{lem_maximum_principle}
        Let $u \in W^{1,2}(V_R)$ be a weak solution of
        \begin{equation}\label{eq_maximum_principle0}
            \left\{
                \begin{array}{lcll}
                    \Delta u            & \geq 0 &\text{in $V_R$}\\
                    \partial_{\nu}u - u & \geq 0 &\text{in $\Gamma_R$},
                \end{array}
            \right.
        \end{equation}
        that is, for all non-negative function $v \in W^{1,2}_0(V_R \cup \Gamma_R)$,
        \begin{equation}\label{eq_maximum_principle1}
            \int_{V_R} \! \langle \nabla u, \nabla v \rangle \dm\LL^n + \int_{\Gamma_R} \! u^* v^* \dm \HH^{n-1} \leq 0.
        \end{equation}
        If $u^* \leq 0$ on $\partial V_R \setminus \Gamma_R$, then $u \leq 0$ on $V_R$.
    \end{lem}

    \begin{proof}
        Let $u \in W^{1,2}(V_R)$ and let $u_+ := p(u)$ where $p \colon \R \to [0,\infty[$ is the orthogonal projection onto $[0,\infty[$.
        According to the chain rule, $u_+ \in W^{1,2}(V_R)$ and for $\LL^n$-a.e.\ $x \in V_R$,
        \begin{align}\label{eq_chain_rule}
            \nabla u_+(x) =
            \begin{cases}
                \nabla u(x) & \text{if $u(x) > 0$}\\
                0           & \text{if $u(x) = 0$}.
            \end{cases}
        \end{align}
        One can also see that $(u_+)^* = p(u^*)$ using the characterization (\ref{eq_char_trace}) and the fact that $p$ is Lipschitz.
        Now, we assume in addition that $u$ is a weak solution of (\ref{eq_maximum_principle0}) and that $u^* \leq 0$ on $\partial V_R \setminus \Gamma_R$.
        We have $u_+ \in W^{1,2}_0(V_R)$ so (\ref{eq_maximum_principle1}) gives
        \begin{equation}
            \int_{V_R} \! \langle \nabla u, \nabla u_+ \rangle \dm \LL^n + \int_{\Gamma_R} \! u^* u_+^* \dm\HH^{n-1} \leq 0.
        \end{equation}
        As $u^* u_+^* = (u_+^*)^2$ and $\langle \nabla u, \nabla u_+ \rangle = \abs{\nabla u_+}^2$, we conclude that $u_+ = 0$ on $V_R$.
    \end{proof}

    \subsection{Hölder continuity up to the boundary}

    \begin{lem}[Hölder continuity]\label{lem_holder}
        Let $u \in W^{1,2}(V_R) \cap L^\infty(V_R)$ be a weak solution of (\ref{eq_robin_problem}).
        There exists constants $C \geq 1$ (depending on $n$, $\alpha$, $A$) and $0 < \sigma < 1$ (depending on $n$) such that for all $x, y \in
        V_R$,
        \begin{equation}
            \abs{u(x) - u(y)} \leq C \abs{u}_\infty \left(\frac{\abs{x - y}}{r}\right)^{\sigma}.
        \end{equation}
        where $r := \max \set{\mathrm{d}(x,\R^n \setminus B_R), \mathrm{d}(y,\R^n \setminus B_R)}$.
    \end{lem}

    The Hölder continuity relies on a weak Harnack inequality at the boundary.
    We temporarily redefine the notation $V_R$, $\Gamma_R$ in the next lemma because is more convenient to work with cylinders rather than balls.

    \begin{lem}[Weak Harnack Inequality]\label{lem_harnack}
        We fix a radius $0 < R \leq 1$.
        We fix a vector $e_n \in \mathbf{S}^{n-1}$ and we decompose each point $x \in \R^n$ as $x = x' + x_n e_n$ where $x' \in e_n^\perp$ and $x_n \in \R$.
        We fix a $1$-Lipschitz function $f \colon e_n^\perp \to \R$ and we assume that $0 \leq f \leq \delta R$ for a certain $0 < \delta \leq \frac{1}{2}$ small enough (depending on $n$).
        Finally, we define
        \begin{align}
            V_R         &:= \set{x \in \R^n | \abs{x'} < R,\ f(x') < x_n < 2R}\\
            \Gamma_R    &:= \set{x \in \R^n | \abs{x'} < R,\ x_n = f(x')}
        \end{align}
        Let $u \in W^{1,2}(V_R)$ be a non-negative weak solution of
        \begin{equation}
            \left\{
                \begin{array}{lcll}
                    \Delta u            & = & 0 &\text{in $V_R$}\\
                    \partial_{\nu}u - u & = & 0 &\text{in $\Gamma_R$}.
                \end{array}
            \right.
        \end{equation}
        Then there exists a constant $C \geq 1$ (depending on $n$) such that
        \begin{multline}\label{eq_harnack}
            \sup \set{u(x) | \abs{x'} \leq \tfrac{1}{2}R,\ 2 \delta R \leq x_n \leq \tfrac{3}{2} R} \\\leq C \inf \set{u(x) | \abs{x'} \leq \tfrac{1}{2}R,\ f(x') < x_n \leq \tfrac{3}{2} R}
        \end{multline}
    \end{lem}

    \begin{proof}
        The letter $C$ is a constant $\geq 1$ that depends on $n$ and whose value might change from one line to the other.
        We introduce the closed cubes
        \begin{align}
            Q   &:= \set{x \in \R^n | \abs{x'} \leq \tfrac{3}{4}R,\ 0 \leq x_n \leq 2\delta R}\\
            Q_i &:= \set{x \in \R^n | \abs{x'} \leq \tfrac{3}{4}R,\ 2\delta R \leq x_n \leq \tfrac{3}{2}R}.\label{eq_Q_0}
        \end{align}
        It suffices to prove that $\sup_{Q_i} u \leq C \inf_{Q_i} u$ and that
        $I \leq C J$ where
        \begin{align}
            I &:= \inf \set{u(x) | \abs{x'} \leq \tfrac{3}{4}R,\ x_n = 2\delta R}\\
            J &:= \inf \set{u(x) | \abs{x'} \leq \tfrac{1}{2}R,\ f(x') < x_n \leq 2\delta R}.
        \end{align}

        We show that $\sup_{Q_i} u \leq C \inf_{Q_i} u$ by applying the interior Harnack inequality for harmonic functions and by observing that
        \begin{equation}
            \mathrm{d}(Q_i,\R^n \setminus V_R) \geq \min \set{\tfrac{1}{4} R,\delta R}.
        \end{equation}
        Here, remember that $\delta$ will be chosen in function of $n$ only.
        Now, we focus on $Q$.
        We define two subsets of $Q$,
        \begin{align}
            W_R         &:= \set{x \in \R^n | \abs{x'} \leq \tfrac{3}{4}R,\ f(x') < x_n < 2 \delta R}\\
            \Sigma_R    &:= \set{x \in \R^n | \abs{x'} \leq \tfrac{3}{4}R,\ x_n = f(x')}
        \end{align}
        and we recall that for $\HH^{n-1}$-a.e.\ $x \in \Sigma_R$, the normal vector to $\Gamma_R$ going upward
        \begin{equation}
            \nu(x) = \frac{-\nabla f(x') + e_n}{\sqrt{1+\abs{\nabla f(x')}^2}}.
        \end{equation}
        We are going to build a paraboloid $p$ such that
        \begin{enumerate}[label = (\roman*)]
            \item
                $p \geq 0$ on $Q$;

            \item
                $p \geq 1$ on $\set{x \in \R^n | \abs{x'} = \tfrac{3}{4}R,\ 0
                \leq x_n \leq 2\delta R}$;

            \item
                $p \leq \frac{99}{100}$ on $\set{x \in \R^n | \abs{x'} \leq
                \tfrac{1}{2}R,\ 0 \leq x_n \leq 2\delta R}$;

            \item
                $\Delta p \leq 0$ on $Q$;

            \item
                $\partial_{\nu} p \leq -1$ on $\Sigma_R$.
        \end{enumerate}
        Once $p$ is built, we apply the maximum principle to $w := u + I p - I$.
        More precisely, one can check that $\Delta w \leq 0$ on $W_R$ and $\partial_{\nu} w \leq u - I \leq w$ on $\Sigma_R$.
        One can also check that $w \geq 0$ on $\partial W_R \setminus \Sigma_R$.
        The maximum principles implies that $w \geq 0$ on $W_R$ and in particular, $u \geq \frac{1}{100}I$ on $\set{x \in \R^n | \abs{x'} \leq \tfrac{1}{2}R,\ f(x') < x_n \leq 2 \delta R}$.
        The paraboloid is
        \begin{equation}
            p(x) := \frac{1}{2} - \left(\frac{x_n}{8\delta R}\right) - \left(\frac{x_n}{4\delta R}\right)^2 + \left(\frac{4\abs{x'}}{3R}\right)^2.
        \end{equation}
        Let us check these properties.
        For $x \in Q$, we have $0 \leq x_n \leq 2 \delta R$ so
        \begin{equation}
            \left(\frac{4\abs{x'}}{3}\right)^2 \leq p \leq \frac{1}{2} + \left(\frac{4\abs{x'}}{3R}\right)^2.
        \end{equation}
        The first three items follow.
        We compute
        \begin{equation}
            \Delta p = -\frac{1}{8\delta^2 R^2} + \frac{32(n-1)}{9 R^2}
        \end{equation}
        and we take $\delta$ small enough (depending on $n$) so that $\Delta p \leq 0$ on $Q$.
        For $x \in \R^n$, we have
        \begin{equation}
            \nabla p(x) = -\frac{e_n}{8\delta R} - \frac{x_n e_n}{8 \delta^2 R^2} + \frac{32 x'}{9 R^2}.
        \end{equation}
        so for $\HH^{n-1}$-a.e.\ $x \in \Sigma_R$,
        \begin{equation}
            \partial_{\nu} p(x) = \frac{1}{\sqrt{1 + \abs{\nabla f(x')}^2}} \left(-\frac{1}{8\delta R} - \frac{x_n}{8 \delta^2 R^2} - \frac{32(x' \cdot \nabla f(x'))}{9 R^2}\right).
        \end{equation}
        In addition $x_n \geq 0$, $\abs{x'} \leq \tfrac{3}{4} R$ and $f$ is $1$-Lipschitz so
        \begin{equation}
            -\frac{1}{8\delta R} - \frac{x_n}{8 \delta^2 R^2} - \frac{32(x' \cdot \nabla f(x'))}{9 R^2} \leq -\frac{1}{8 \delta R} + \frac{8}{3 R}.
        \end{equation}
        We take $\delta$ small enough so that $-\frac{1}{8 \delta R} + \frac{8}{3 R} \leq -\frac{\sqrt{2}}{R}$ and thus $\partial_{\nu}p \leq -\frac{1}{R} \leq -1$ (because $R \leq 1$).
    \end{proof}

    We are ready to prove Lemma~\ref{lem_holder} (we come back to the notations $V_R$, $\Gamma_R$ of the introduction).
    \begin{proof}[Proof of Lemma~\ref{lem_holder}]
        We start with a preliminary observation.
        There exists a radius $0 < r_0 \leq 1$ (depending on $n$, $\alpha$, $A$) such that the following holds true.
        For all $x \in \Gamma_R$, for all $0 \leq r \leq r_0$ such that $B(x,4r) \subset B_R$, there exists a vector $e_n \in \mathbf{S}^{n-1}$ inducing a coordinate system $y = x + (y' + y_n e_n)$ and a $\tfrac{1}{2}\delta$-Lipschitz function $f \colon e_n^\perp \to \R$ (where $0 < \delta \leq \tfrac{1}{2}$ is the constant of Lemma~\ref{lem_harnack}) such that $f(0) = 0$ and
        \begin{subequations}\label{eq_V_graph}
        \begin{align}
            V_R \cap B(x,3r)        &= \set{y \in B(x,3r) | y_n > f(y')}\\
            \Gamma_R \cap B(x,3r)   &= \set{y \in B(x,3r) | y_n = f(y')}.
        \end{align}
        \end{subequations}
        Note that $\abs{f} \leq \tfrac{1}{2} \delta r$ in $e_n^\perp \cap B(0,r)$.
        We consider the translated coordinate system whose origin is $x - \tfrac{1}{2}\delta r e_n$. In the new system, every point $y \in \R^n$ can be written
        \begin{equation}
            y = x - \tfrac{1}{2} \delta r e_n + (z' + z_n e_n).
        \end{equation}
        Equivalently, the new coordinates are $z' = y'$ and $z_n = y_n + \tfrac{1}{2} \delta r$.
        Setting $g = f + \tfrac{1}{2} \delta r$, we have $0 \leq g \leq \delta r$ and the description (\ref{eq_V_graph}) implies
        \begin{subequations}
        \begin{multline}
            V_R \cap \set{y \in \R^n | \abs{z'} < r,\ z_n < 2r} \\
            = \set{y \in \R^n | \abs{z'} < r,\ g(z') < z_n < 2r},
        \end{multline}
        \begin{multline}
            \Gamma_R \cap \set{y \in \R^n | \abs{z'} < r,\ z_n < 2r} \\
            = \set{y \in \R^n | \abs{z'} < r,\ z_n = g(y')}.
        \end{multline}
        \end{subequations}
        Therefore, we can apply Lemma~\ref{lem_harnack} in the new coordinates.
        Whenever a non-negative function $v \in W^{1,2}(V_R \cap B(x,4r))$ is a weak solution of
        \begin{equation}
            \left\{
                \begin{array}{lcl}
                    \Delta v            & = & 0 \quad \text{in $V_R \cap B(x,4r)$}\\
                    \partial_{\nu}v - v & = & 0 \quad \text{in $\Gamma_R \cap B(x,4r)$},
                \end{array}
            \right.
        \end{equation}
        we can draw the conclusion that there exists $C_0 > 1$ (depending on $n$) such that
        \begin{equation}
            v(x + r e_n) \leq C_0 \inf \set{v(y) | y \in V_R \cap B(x,\tfrac{1}{2} r)}.
        \end{equation}

        Now we define
        \begin{equation}
            \Delta := \set{(x,r) | x \in \overline{V_R},\ 0 < r \leq r_0,\ B(x,r) \subset B_R}
        \end{equation}
        and for $(x,r) \in \Delta$,
        \begin{equation}
            \textrm{osc}(x,r) := \sup \set{\abs{u(z) - u(y)} | y,z \in V_R \cap B(x,r)}.
        \end{equation}
        We are going to show that there exists $L > 1$ (close to $1$, depending on $n$) such that for all $(x,r) \in \Delta$,
        \begin{equation}\label{eq_o_goal}
            \textrm{osc}(x,\tfrac{1}{32}r) \leq L^{-1} \textrm{osc}(x,r).
        \end{equation}
        We are going to distinguish three cases: the case $x \in \Gamma_R$, the case $\Gamma_R \cap B(x,\tfrac{1}{12}r) \ne \emptyset$ and the case $B(x,\tfrac{1}{12}r) \subset V_R$.

        Let $(x,r) \in \Delta$ with $x \in \Gamma_R$.
        We define
        \begin{align}
            M &:= \sup \set{u(z) | z \in V_R \cap B(x,r)}\\
            m &:= \inf \set{u(z) | z \in V_R \cap B(x,r)}
        \end{align}
        and
        \begin{align}
            M' &:= \sup \set{u(z) | z \in V_R \cap B(x,\tfrac{1}{8}r)}\\
            m' &:= \inf \set{u(z) | z \in V_R \cap B(x,\tfrac{1}{8}r)}.
        \end{align}
        We apply the preliminary paragraph to $M - u$ and $u - m$ as functions of $W^{1,2}(V_R \cap B(x,r))$ and we obtain that there exists $C_0 > 1$ (depending on $n$) such that
        \begin{align}
            M - u(x + \tfrac{1}{4}r e_n) &\leq C_0(M - M')\\
            u(x + \tfrac{1}{4}r e_n) - m &\leq C_0(m' - m).
        \end{align}
        It follows that $M - m \leq C_0(M - M' + m' - m)$ and then
        \begin{equation}
            M' - m' \leq L^{-1} (M - m)
        \end{equation}
        where $L := C_0(C_0 - 1)^{-1} > 1$.
        This proves that
        $\textrm{osc}(x,\tfrac{1}{8}r) \leq L^{-1} \textrm{osc}(x,r)$.

        Let $(x,r) \in \Delta$ be such that there exists $x^* \in \Gamma_R$ with $\abs{x - x^*} < \tfrac{1}{12}r$.
        We observe that
        \begin{equation}
            B(x,\tfrac{1}{32}r) \subset B(x^*,\left(\tfrac{1}{32}+\tfrac{1}{12}\right)r),
        \end{equation}
        that by the previous step
        \begin{equation}
            \textrm{osc}(x^*,\left(\tfrac{1}{32}+\tfrac{1}{12}\right)r) \leq L^{-1} \textrm{osc}(x^*,\left(\tfrac{1}{4}+\tfrac{2}{3}\right)r)
        \end{equation}
        and finally that
        \begin{equation}
            B(x^*,\left(\tfrac{1}{4}+\tfrac{2}{3}\right)r) \subset B(x,r).
        \end{equation}
        This proves that $\textrm{osc}(x,\tfrac{1}{32}r) \leq L^{-1}
        \textrm{osc}(x,r)$.

        Let $(x,r) \in \Delta$ be such that $B(x,\tfrac{1}{12}r) \subset V_R$.
        We can proceed as in the case $x \in \Gamma_R$ (but replacing Lemma~\ref{lem_harnack} by the interior Harnack inequality for harmonic functions) to get
        \begin{equation}
            \textrm{osc}(x,\tfrac{1}{32}r) \leq L^{-1} \textrm{osc}(x,\tfrac{1}{16}r).
        \end{equation}

        We have proved (\ref{eq_o_goal}) in all cases.
        We fix $x \in V_R$ and we define the radii $r_1 := \mathrm{d}(x, \R^n \setminus B_R)$ and $\rho := \min \set{r_0, r_1}$ so that $(x,\rho) \in \Delta$.
        For all $0 < r \leq \rho$, we have $\textrm{osc}(x,\tfrac{1}{32}r) \leq L^{-1} \textrm{osc}(x,r)$ and it is easy to deduce that for all $0 < r \leq \rho$,
        \begin{equation}
            \textrm{osc}(x,r)  \leq L \left(\frac{r}{\rho}\right)^\sigma \textrm{osc}(x,\rho),
        \end{equation}
        where $\sigma := \frac{\ln(L)}{\ln(32)} > 0$.
        This implies that for all $y \in V_R \cap B(x,\rho)$,
        \begin{align}
            \abs{u(y) - u(x)}   &\leq L \left(\frac{\abs{y-x}}{\rho}\right)^\sigma \textrm{osc}(x,\rho)\\
                                &\leq 2 L \abs{u}_\infty \left(\frac{\abs{y-x}}{\rho}\right)^\sigma
        \end{align}
        In fact, this inequality is also true for $y \in V_R$ such that $\abs{y - x} \geq \rho$ because we always have $\abs{u(y) - u(x)} \leq 2 \abs{u}_\infty$.
        We also underline that $\rho \geq r_0 r_1$ because $r_0 \leq 1$ and $r_1 \leq R \leq 1$.
        The lemma holds true for the constant $C := 2 L r_0^{-\sigma}$ (which depends on $n$ and $r_0$).
    \end{proof}

    \subsection{Gradient estimates}

    According to~\cite[Theorem 1.2]{Zhang}, the viscosity solutions are pointwise $C^{1,\alpha}$ up to the boundary.
    Although we use a weak formulation, the proof also applies in our case because it relies on the maximum principle (Lemma~\ref{lem_maximum_principle}), the Hölder continuity (Lemma~\ref{lem_holder}) and regularity results for solutions of the Neumann problem in a spherical cap.

    \begin{lem}[Schauder Estimate]\label{lem_robin_schauder}
        Let $u \in W^{1,2}(V_R) \cap L^\infty(V_R)$ be a weak solution of (\ref{eq_robin_problem}).
        Then there exists $C \geq 1$ and $0 < \sigma < 1$ (depending on $n$, $\alpha$, $A$) such that
        \begin{equation}
            \abs{\nabla u}_\infty + R^\sigma \left[\nabla u\right]_{\sigma} \leq C R^{-1} \textrm{osc}(u) + C \abs{u}_\infty,
        \end{equation}
        where the left-hand side is computed on $\tfrac{1}{2}V_R$ and the symbol $\textrm{osc}(u)$ means $\sup \set{\abs{u(x) - u(y)} | x,y \in V_R}$
    \end{lem}

    \begin{proof}
        The letter $C$ is a constant $\geq 1$ that depends on $n$, $V$ and whose value might change from one line to the other.
        We fix any $x_0 \in \tfrac{3}{4} V_R$.
        The function $v := u - u(x_0)$ is a weak solution of the Neumann problem
        \begin{equation}
            \left\{
                \begin{array}{lcll}
                    \Delta v               & = & 0  &\text{in $\tfrac{3}{4} V_R$}\\[3pt]
                    \partial_{\nu}v        & = & u  &\text{in $\tfrac{3}{4} \Gamma_R$}.
                \end{array}
            \right.
                \end{equation}
                We have restricted the domain of the equation so that the trace $u^*$ is Hölder on $\tfrac{3}{4} \Gamma_R$.
                More precisely, $\abs{u^*}_\infty \leq \abs{u}_\infty$ and according to Lemma~\ref{lem_holder}, there exists $0 < \sigma < 1$ (depending on $n$) such that for all $x, y \in \tfrac{3}{4} \Gamma_R$,
                \begin{equation}
                    \abs{u^*(x) - u^*(y)} \leq C \abs{u}_\infty \left(\frac{\abs{x-y}}{R}\right)^\sigma.
                \end{equation}
                Then we apply the scaled version of~\cite[Theorem 1.2]{Zhang}, assuming $\sigma$ small enough if necessary (depending on $n$, $\alpha$, $A$). We obtain
                \begin{equation}
                    R \abs{\nabla v}_\infty + R^{1+\sigma} \left[\nabla v\right]_{\sigma} \leq C \abs{v}_\infty + C R \abs{u}_\infty,
                \end{equation}
                where the left-hand side is computed on $\tfrac{1}{2}V_R$.
            \end{proof}

            The last result allows to control the oscillations of $u$.

            \begin{lem}[Oscillations estimate]\label{lem_robin_boundedness}
                Let $u \in W^{1,2}(V_R) \cap L^\infty(V_R)$ be a weak solution of (\ref{eq_robin_problem}).
                Then there exists $C \geq 1$ (depending on $n$, $\alpha$, $A$) such that
                \begin{equation}
                    \textrm{osc}(u) \leq C R \left(\fint_{V_R} \abs{\nabla u}^2 \dm\LL^n\right)^{\frac{1}{2}} + C R \abs{u}_\infty,
                \end{equation}
                where $\textrm{osc}(u) := \sup \set{\abs{u(x) - u(y)} | x,y \in \tfrac{1}{2} V_R}$.
            \end{lem}

            \begin{proof}
                The letter $C$ is a constant $\geq 1$ that depends on $n$, $\alpha$, $A$ and whose value might change from one line to the other.
                Let $m$ be the average value of $u$ on $V_R$.
                The function $v := u - m$ is a weak solution of the Neumann problem
                \begin{equation}
                    \left\{
                        \begin{array}{lcll}
                            \Delta v               & = & 0 &\text{in $V_R$}\\
                            \partial_{\nu}v        & = & u &\text{in $\Gamma_R$}.
                        \end{array}
                    \right.
                \end{equation}
                We apply a local boundedness estimate for weak solutions of Neumann problems (\cite[Theorem 1.6 and Remark 1.12]{Kim}). We obtain
                \begin{equation}\label{eq_neumann_bound}
                    \abs{v}_\infty \leq C \left(\fint_{V_R} \! \abs{v}^2 \dm \LL^n\right)^{\frac{1}{2}} + C R \abs{u}_\infty,
                \end{equation}
                where the left-hand side is computed on $\tfrac{1}{2} V_R$.
                The triangular inequality shows that
                \begin{equation}
                    \textrm{osc}(u) \leq 2 \abs{v}_\infty
                \end{equation}
                and the Poincar\'e inequality gives
                \begin{equation}
                    \left(\fint_{V_R} \! \abs{v}^2 \dm \LL^n\right)^{\frac{1}{2}} \leq C R \left(\fint_{V_R} \! \abs{\nabla u}^2 \dm \LL^n\right)^{\frac{1}{2}}.
                \end{equation}
            \end{proof}

            \section{Extracts from David's book}\label{appendix_david}

            We extract some important results of~\cite{David}.
            We work in an open set $X$ of the Euclidean space $\R^n$ ($n > 1$) and we fix a triple of parameters $\mathcal{P} := (r_0,a,M)$ composed of $r_0 > 0$, $a \geq 0$ and $M \geq 1$.
            We start by summarizing Definitions $2.1$ (\emph{admissible pairs}), $7.2$ (\emph{competitors}), $7.21$ (\emph{local quasiminimizers}) and $8.24$ (\emph{coral pairs}) 

            \begin{defi}\label{defi_david_quasi}
                The set of \emph{admissible pairs} $\mathcal{A}$ is the set of all pairs $(u,K)$ where $K \subset X$ is relatively closed in $X$ and $u \in W^{1,2}_{\loc}(X \setminus K)$.
                Let $(u,K)$ be an admissible pair and let $B$ be an open ball such that $\overline{B} \subset X$.
                A \emph{competitor} of $(u,K)$ in $B$ is a pair $(v,L) \in \mathcal{A}$ such that $K \setminus \overline{B} = L \setminus \overline{B}$ and $u = v$ $\LL^n$-a.e.\ on $X \setminus (K \cup \overline{B})$.
                In this case, we set
                \begin{equation}
                    E(u) := \int_B \! \abs{\nabla u}^2 \dm\LL^n,\quad E(v) := \int_B \! \abs{\nabla v}^2
                \end{equation}
                and
                \begin{equation}
                    \Delta E := \max \set{(E(v) - E(u)), M(E(v) - E(u))}.
                \end{equation}
                We say that $(u,K)$ is a \emph{local $\mathcal{P}$-quasiminimizer in $X$} if for all open balls $B$ of radius $0 < r \leq r_0$ such that $\overline{B} \subset X$, for all competitors $(v,L)$ of $(u,K)$ in $B$, we have
                \begin{equation}\label{eq_quasi}
                    \HH^{n-1}(K \setminus L) \leq M \HH^{n-1}(L \setminus K) + \Delta E + a r^{n-1}.
                \end{equation}
                In addition, we say that $(u,K)$ is \emph{coral} if $K = \spt(\HH^{n-1} \mres K)$ in $X$.
                This means that for all $x \in K$ and all $r > 0$, $\HH^{n-1}(K \cap B(x,r)) > 0$.
            \end{defi}

            \begin{rmk}\label{rmk_david}
                If $(u,K)$ is a quasiminimizer, we can see that $K$ is $\HH^{n-1}$ locally finite. For all open ball $B$ of radius $0 \leq r \leq r_0$ such that $\overline{B} \subset X$, we consider the competitor
                \begin{equation}
                    v :=
                    \begin{cases}
                        u   &   \text{in} \ X \setminus \overline{B}\\
                        0   &   \text{in} \ \overline{B}
                    \end{cases}
                \end{equation}
                and $L = (K \setminus \overline{B}) \cup \partial B$.
                In particular, we have $K \setminus L = K \cap B$ and $L \setminus K \subset \partial B$ and $\Delta E \leq 0$.
                This proves that $\HH^{n-1}(K \cap \overline{B}) \leq ((M+1) n \omega_n + a)r^{n-1}$.

                For all competitors $(v,L)$ of $(u,K)$, we have either $\HH^{n-1}(L \setminus K) = \infty$ and thus (\ref{eq_quasi}) says nothing or $\HH^{n-1}(K \setminus L) < \infty$ and thus $L$ is $\HH^{n-1}$ locally finite in $X$.
                This justifies that we can assume that $L$ is $\HH^{n-1}$ locally finite without loss of generality.
            \end{rmk}

            We are mainly concerned about Ahlfors-regularity (Definition 18.9) and uniform rectifiability (Section 73).

            \begin{defi}[Ahlfors-regularity]
                A closed set $E \subset \R^n$ is Ahlfors-regular of dimension $n-1$ if there exists a constant $C \geq 1$ such that for all $x \in E$ and for all $0 < r < \mathrm{diam}(E)$
                \begin{equation}
                    C^{-1} r^{n-1} \leq \HH^{n-1}(E \cap B(x,r)) \leq C r^{n-1}.
                \end{equation}
            \end{defi}

            We don't give definitions of uniform rectifiability because there are too many.
            They are equivalent for closed, Ahlfors-regular sets.
            The reader can find a survey of uniform rectifiability in~\cite[Section 73]{David} and also on Guy David's webpage (Notes-Parkcity.dvi).

            Next, we summarize Definition 18.14 (\emph{TRLQ class}), Section 72 and Section 74.
            It says that quasiminimizers are locally Ahlfors-regular and locally contained in a uniformly rectifiable set.
            In fact, we have already seen the first item in Remark~\ref{rmk_david}.

            \begin{thm}\label{thm_david}
                Let $\mathcal{P} := (r_0,a,M)$ be a triple of parameters composed of $r_0 > 0$, $a \geq 0$ and $M \geq 1$.
                Assume that $a$ is small enough (depending on $n$, $M$). 
                Let $(u,K)$ be a coral and local $\mathcal{P}$-quasiminimizer in $X$.
                \begin{enumerate}
                    \item
                        For all $x \in X$, for all $0 < r \leq r_0$ such that $B(x,r) \subset X$,
                        \begin{equation}
                            \HH^{n-1}(K \cap B(x,r)) \leq C r^{n-1}.
                        \end{equation}
                        where $C \geq 1$ depends on $n$, $M$.

                    \item
                        For all $x \in K$, for all $0 < r \leq r_0$ such that $B(x,r) \subset X$,
                        \begin{equation}
                            \HH^{n-1}(K \cap B(x,r)) \geq C^{-1} r^{n-1}.
                        \end{equation}
                        where $C \geq 1$ depends on $n$, $M$.

                    \item
                        For all $x \in K$ and $0 < r < r_0$ such that $B(x,r) \subset X$, there is a closed, Ahlfors-regular, uniformly rectifiable set $E$ of dimension $(n-1)$ such that $K \cap \tfrac{1}{2} B(x,r) \subset E$.
                        The constants for the Ahfors-regularity and uniform rectifiability depends on $n$, $M$ and $a$.
                \end{enumerate}
            \end{thm}

            \begin{rmk}\label{rmk_david2}
                One can observe that (\ref{eq_quasi}) implies
                \begin{equation}\label{eq_quasi2}
                    \HH^{n-1}(K \cap \overline{B}) \leq M \HH^{n-1}(L \cap \overline{B}) + \Delta E + a r^{n-1}.
                \end{equation}
                This is equivalent when $M = 1$ but strictly weaker when $M > 1$.

                We claim that Theorem~\ref{thm_david} still holds true with (\ref{eq_quasi2}) in place of (\ref{eq_quasi}).
                The first item is easy (see Remark~\ref{rmk_david}).
                The second item works as usual.
                The most critical point is the third item.
                In Section 74, David builds a suitable competitor $(w,G)$ of $(u,K)$ in a ball $B$.
                The set $G$ is of the form $G = (K \setminus B) \cup Z$ where $Z$ a special subset of $\partial B$ containing $K \cap \partial B$.
                The quasi-minimality condition (\ref{eq_quasi}) is used only once at line (22) of Section 74.
                Then David uses the inequalities
                \begin{align}
                    \HH^{n-1}(K \setminus G)    &\geq \HH^{n-1}(K \cap B)\\
                    \HH^{n-1}(G \setminus K)    &\leq \HH^{n-1}(Z)
                \end{align}
                but we also have anyway
                \begin{align}
                    \HH^{n-1}(K \cap \overline{B}) &\geq \HH^{n-1}(K \cap B)\\
                    \HH^{n-1}(G \cap \overline{B}) &\leq \HH^{n-1}(Z).
                \end{align}
            \end{rmk}
        \end{appendices}

        \section*{Acknowledgments}

        This work was co-funded by the European Regional Development Fund and the Republic of Cyprus through the Research and Innovation Foundation (Project: EXCELLENCE/1216/0025).

    \vspace{2em}

    \begin{tabular}{l}
        Camille Labourie \\ University of Cyprus \\ Department of Mathematics\\ \& Statistics \\ P.O. Box 20537\\
        Nicosia, CY- 1678 CYPRUS
        \\ {\small \tt labourie.camille@ucy.ac.cy}
    \end{tabular}
    \begin{tabular}{lr}
        Emmanouil Milakis\\ University of Cyprus \\ Department of Mathematics\\ \& Statistics \\ P.O. Box 20537\\
        Nicosia, CY- 1678 CYPRUS
        \\ {\small \tt emilakis@ucy.ac.cy}
    \end{tabular}
    \end{document}